\documentclass[10pt]{article}

\usepackage{pdfpages}
\usepackage{amsfonts}
\usepackage{graphicx}
\usepackage{epstopdf}
\usepackage{algorithm}
\usepackage{algpseudocode}
\usepackage{subfigure}

\usepackage{amsthm}
\usepackage{amsmath}
\usepackage{xr-hyper}
\RequirePackage[colorlinks,linkcolor=blue,citecolor=blue,urlcolor=magenta]{hyperref}
\usepackage{cleveref}

\ifpdf
  \DeclareGraphicsExtensions{-eps-converted-to.pdf,.pdf,.png,.jpg}
\else
  \DeclareGraphicsExtensions{-eps-converted-to.pdf}
\fi

\usepackage{enumitem}
\setlist[enumerate]{leftmargin=.5in}
\setlist[itemize]{leftmargin=.5in}

\usepackage{color}

\def\videbox{\mathbin{\vbox{\hrule\hbox{\vrule height1.4ex \kern.6em\vrule height1.4ex}\hrule}}}
\def\demend{\hfill $\videbox$\\}

\newcommand{\bigO}{\ensuremath{\mathcal O}}
\newcommand{\argminE}{\mathop{\mathrm{argmin}}}
\newcommand{\argmax}{\mathop{\mathrm{argmax}}}

\newcommand{\R}{\ensuremath{\mathbb R}}

\newcommand{\EE}{\ensuremath{\mathbb E}}
\newcommand{\TT}{\ensuremath{\mathbb T}}

\newcommand{\NN}{\ensuremath{\mathbb N}}

\newcommand{\XX}{\ensuremath{\mathcal X}}
\newcommand{\YY}{\ensuremath{\mathcal Y}}
\newcommand{\ZZ}{\ensuremath{\mathbb Z}}

\newcommand{\CC}{\ensuremath{\mathbb C}}

\newcommand{\Sc}{\ensuremath{\mathbb S}}
\newcommand{\Bb}{\ensuremath{\mathbb B}}
\newcommand{\ee}{\ensuremath{\varepsilon}}
\newcommand{\SL}{\sum_{\lambda \in \Lambda}}
\newcommand{\tht}{\theta}

\newcommand*{\supl}{\operatornamewithlimits{sup}\limits}

\newcommand{\thefont}[2]{\fontsize{#1}{#2}\fontshape{n}\selectfont}
\newcommand{\1}{\rlap{\thefont{10pt}{12pt}1}\kern.16em\rlap{\thefont{11pt}{13.2pt}1}\kern.4em}

\newtheorem{thm}{Theorem}[section]

\newtheorem{prop}{Proposition}[section]

\newtheorem{lem}{Lemma}[section]
\newtheorem{definition}{Definition}[section]
\newtheorem{hyp}{Assumption}[section]
\crefname{section}{Section}{Section}

\title{Stochastic optimal transport in Banach spaces for regularized estimation of multivariate quantiles}

\author{ Bernard Bercu \\ bernard.bercu@math.u-bordeaux.fr \and J\'{e}r\'{e}mie Bigot \\ jeremie.bigot@math.u-bordeaux.fr \and Gauthier Thurin \\ gauthier-louis.thurin@math.u-bordeaux.fr \thanks{Universit\'e de Bordeaux, Institut de Math\'ematiques de Bordeaux et CNRS (UMR 5251)} }

\usepackage{amsopn}

\usepackage{geometry}
\date{}

\begin{document}

\maketitle

\begin{abstract}
We introduce a new stochastic algorithm for solving entropic optimal transport (EOT) between two absolutely continuous probability measures $\mu$ and $\nu$. Our work is motivated by the specific setting of Monge-Kantorovich quantiles where the source measure $\mu$ is either the uniform distribution on the unit hypercube or the spherical uniform distribution. Using the knowledge of the source measure, we propose to parametrize a Kantorovich dual potential by its Fourier coefficients. In this way, each iteration of our stochastic algorithm reduces to two Fourier transforms that enables us to make use of the  Fast Fourier Transform (FFT) in order to implement a fast numerical method to solve EOT. We study the almost sure convergence of our stochastic algorithm that takes its values in an infinite-dimensional Banach space. Then, using numerical experiments, we illustrate the performances of our approach on the computation of regularized Monge-Kantorovich quantiles. In particular, we investigate the potential benefits of  entropic regularization for the smooth estimation of multivariate quantiles using data sampled from the target measure $\nu$.
\end{abstract}

\



\section{Introduction}

Consider  a probability distribution $\nu$ supported on a subset $\YY \subset \R^d$. In the scalar case $d=1$, the quantile function of $\nu$   is nothing else than the generalized inverse  $F_{\nu}^{-1}$ of the cumulative distribution function $F_{\nu}$ of $\nu$. However, in the multi-dimensional case $d \geq 2$, there does not exist a standard notion of multivariate quantiles as there is no canonical ordering in $\R^d$. Therefore, various notions  of quantiles in dimension $d \geq 2$ have been proposed in the statistical literature, some of them being inspired by the notion of  data depth introduced in \cite{tukey75}  and other based on geometric principles \cite{Chaudhuri96}. We refer the reader to Section 1.2 in \cite{Hallin-AOS_2021} for a recent survey of the many existing concepts of multivariate quantiles.

The aim of this paper is to investigate the notion of Monge-Kantorovich (MK) quantiles using the theory of quadratic optimal transport (OT) that has been introduced in \cite{chernozhukov2015mongekantorovich}. The basic concepts of MK quantiles can be summarized as follows. 
For $\mathcal{P}_d$ the set of Lebesgue-absolutely continuous probability measures on $\R^d$, one first considers a reference distribution $\mu\in \mathcal{P}_d$, supported on a convex and compact set $\XX \subset \R^d$.
As discussed in  \cite{chernozhukov2015mongekantorovich}, this reference measure $\mu$ is typically either the uniform distribution on $\XX = [0,1]^d$ or the {\it spherical uniform}\footnote{ {\it Spherical uniform} refers to the distribution $\mu_{S}$ of a random vector $X=R\Phi$ where $R$ and $\Phi$ are independent and drawn uniformly from [0,1] and the unit hypersphere $\Sc^{d-1} = \{\varphi \in \R^d \; : \; \|\varphi \| = 1\}$, respectively.} distribution on the unit ball. 
Then, the MK quantile function of a square integrable probability measure $\nu$, with respect to $\mu$, is defined as the  optimal transport map $Q : \XX \to \YY$ between $\mu$ and $\nu$. More precisely, let $X$ be a random vector with distribution $\mu$. Then, $Q$ is the optimal mapping satisfying
\begin{equation}
\label{DEFQ}
Q =\argminE_{T \; : \; T \# \mu = \nu} \EE \Big( \frac{1}{2} \| X - T(X) \|^2 \Big),
\end{equation}
the notation $T \# \mu = \nu$ meaning that $T : \XX \to \YY$ is a push-forward map from $\mu$ to $\nu$, and $\| \cdot \|$ standing for the usual Euclidean norm in $ \R^d$. 
There, one can rely on the well-known Kantorovich duality (see e.g.\ \cite{santambrogio2015optimal,Villani})  of optimal transport to characterize $Q$.   Since $\mu$ is absolutely continuous, it is well-known \cite{Brenier91,cuesta1989} that  $Q$ can be rewritten as
\begin{equation}
Q(x) = x - \nabla u_{0}(x) = \nabla \psi_{0}(x)  \quad \mbox{with} \quad \psi_{0}(x) =x - u_0(x),
\label{DEFQgrad}
\end{equation}
for $\mu$-almost every $x \in \XX$. In the above equation, $u_{0}$ denotes the unique solution, up to a scalar translation, of the Kantorovich dual formulation of OT
\begin{equation}
u_{0} \in \argmax_{u \in L^{1}(\mu)} \int_{\XX} u(x) d\mu(x) +  \int_{\YY} u^{c}(y) d\nu(y), \label{eq:dualOT}
\end{equation}
where $u^{c} : \YY \to \R$ is the $c$-conjugate of a function $u \in L^{1}(\mu)$ in the sense that
\begin{equation*}
u^{c}(y) = \inf_{x \in \XX} \left\{ c(x,y) - u(x) \right\} \hspace{1cm} \mbox{with}  \hspace{1cm} c(x,y) = \frac{1}{2} \|x-y\|^2.
\end{equation*}
Based on a sample $(Y_{1},\ldots,Y_{n})$ from $\nu$, it is natural to estimate
$Q$ by the plug-in estimator
\begin{equation}
\label{DEFQn0}
\widehat{Q}_n =\argminE_{T \; : \; T \# \mu = \widehat{\nu}_{n}} \EE \Big( \frac{1}{2} \| X - T(X) \|^2 \Big) 
\hspace{1cm} \mbox{where} \hspace{1cm} 
\widehat{\nu}_{n} = \frac{1}{n} \sum_{j=1}^{n} \delta_{Y_{j}}.
\end{equation}
Alternatively, one has from \eqref{DEFQgrad} and \eqref{eq:dualOT} that for all $x\in \XX$, $\widehat{Q}_n(x) = x - \nabla \widehat{u}_{n}(x)$ where
\begin{equation}
\widehat{u}_{n} \in \argmax_{u \in L^{1}(\mu)} \int_{\XX} u(x) d\mu(x) +  \int_{\YY} u^c(y) d\widehat{\nu}_{n}(y). \label{eq:hatun}
\end{equation}
Finding a numerical solution to the problem \eqref{eq:hatun} involves the use of optimization techniques in the Banach space $L^{1}(\mu)$ which is a delicate issue that is tackled in the present paper. More precisely, we propose a new stochastic algorithm in order to estimate the dual potential $u_0$  using computational  optimal transport \cite{peyre2020computational} based on entropic regularization  \cite{cuturi2013sinkhorn}, which also yields a new regularized estimator of the MK quantile function $Q$.

In the last years, the benefit of this regularization has been to allow the use of OT based methods in statistics and machine learning.
In this paper, we also advocate the use of entropic OT (EOT) to obtain an estimator of the  dual potential $u_0$  that is smoother than $\widehat{u}_{n}$, leading to an estimator of the MK quantile function $Q$ that is also smoother than $\widehat{Q}_n$. More precisely, we recall that the dual formulation of EOT as formulated in  \cite{aude2016stochastic} is
\begin{equation}
\max_{u \in L^{1}(\mu)} \int_{\XX} u(x) d\mu(x) +  \int_{\YY} u^{c,\varepsilon}(y) d\nu(y) -\ee \label{eq:dualOTreg}
\end{equation}
where $\varepsilon \geq 0$ stands for a regularization parameter and $u^{c,\varepsilon}$ is the smooth conjugate of $u \in L^{1}(\mu)$ defined, for $\varepsilon > 0$, by
\begin{equation}
u^{c,\ee}(y) = -\ee \log \left( \int_{\XX} \exp \Big( \frac{u(x) -c(x,y)}{\varepsilon} \Big) d\mu(x) \right) 
 \label{eq:reg-c-transform}
\end{equation}
and $u^{c,0}(y) =u^{c}(y)$.
In what follows, a function that can be expressed as a smooth conjugate will be called a {\it regularized $c$-transform}. 
The quadratic cost function $c$ belongs to $L^1(\mu \otimes \nu)$ as soon as $\nu$ has a finite second moment.
Thus, it is known that, up to an additive constant, the solution of \eqref{eq:dualOTreg} is unique for any $\varepsilon > 0$, see e.g.\ the discussion in  \cite[Section 2]{bercu2020asymptotic}. 
The estimation of such a solution is the target of this work. To this end, we mainly focus on the setting where $\mu$ is the uniform distribution on $\XX = [0,1]^d$, and we parametrize a dual function $u \in L^{1}(\mu)$ by its decomposition in the standard Fourier basis $\phi_\lambda(x)=e^{2\pi i \langle \lambda,x \rangle}$, for $\lambda \in \ZZ^d$, that is
\begin{equation}\label{uSF}
u(x) = \sum_{\lambda \in \Lambda} \theta_\lambda \phi_\lambda(x).
\end{equation}
Based on a sample $(Y_1,\ldots,Y_n)$ from $\nu$, we estimate the Fourier coefficients $\theta=(\theta_\lambda)_{\lambda \in \Lambda} $ via a stochastic algorithm $\widehat{\theta}_{n}=(\widehat{\theta}_{n,\lambda})_{\lambda \in \Lambda} $, which allows us to propose a natural plug-in estimator
\begin{equation*}
\widehat{u}_{\varepsilon}^{\, n}(x) = \SL \widehat{\theta}_{n,\lambda} \phi_\lambda(x). 
\end{equation*}
An estimator of $Q$ is then induced from the entropic analog of \eqref{DEFQgrad} using a regularized $c$-transform of $\widehat{u}_{\varepsilon}^{\, n}$, and the notion of barycentric projection (see e.g.\  \cite[Section 3]{pooladian2021entropic}).
From a computational point of view, our stochastic algorithm, described in  \Cref{sec:Fourier_coeffs}, mainly involves the use of two Fast Fourier Transforms (FFT) and the choice of a regular grid of $p$ points in $\XX$ to estimate a set of $p$ Fourier coefficients.  The computational cost of our recursive procedure at each iteration is thus of order $\bigO\left( p \log(p)\right)$. Therefore, its numerical cost, at each iteration, is {\it independent} of the sample size $n$ for which multivariate quantiles need to be computed.

\subsection{Relation to previous works}

\subsubsection{Comparison to other algorithms for solving OT}

The estimation of $Q$ using the plug-in estimator $\widehat{Q}_{n}$ based on the empirical measure $\widehat{\nu}_{n}$ can begin 
with various computational strategies to solve OT between $\mu$ and $\widehat{\nu}_{n}$.
One can replace $\mu$ by a discrete measure $\widehat{\mu}_{n}$ on a regular grid and then solve a
\textit{discrete} OT problem between $\widehat{\mu}_{n}$ and $\widehat{\nu}_{n}$ as in \cite{chernozhukov2015mongekantorovich,Hallin-AOS_2021}. However, the computational cost of such a discrete OT problem is potentially very high 
because it scales cubically in the number of observations \cite{peyre2020computational}.  It is also proposed in \cite{ghosal2021multivariate} to compute the semi-dual problem \eqref{eq:hatun} using the Newton-type algorithms proposed in \cite{merigot18}.

In the present paper, we suggest a new strategy which relies on a parametrization of the dual function $u$ by its Fourier coefficients, which allows us to better make use of the knowledge of the reference distribution $\mu$.
Beyond the context of multivariate quantiles, the estimation of OT maps is an active area of research. 
Dual potentials were parameterized by wavelets expansions in \cite{hutter_rigollet}, and another popular approach is based on neural networks, see $e.g.$ \cite{bunne2022supervised,korotin2021wasserstein,makkuva2020optimal}. Other recent contributions on the estimation of OT maps also include  \cite{deb2021rates,Manole2021,muzellec2021near,vacher2024optimal}. 
In \cite{pooladian2021entropic}, the entropic map has been studied as a natural alternative with respect to entropic regularization, and we follow this line of work in the quantiles' context. 

Stochastic algorithms for solving the {\it semi-discrete} OT problem \eqref{eq:hatun} betwen an absolutely continuous measure $\mu$ and the empirical measure $ \widehat{\nu}_{n}$  have already been proposed in \cite{bercu2020asymptotic,BB_GN,aude2016stochastic}. 
Dual functions $v \in L^{1}( \widehat{\nu}_{n})$ can be identified to their values $v(Y_i)$ for $1\leq i \leq n$, which yields, for $\ee\geq0$, the following stochastic optimization problem
\begin{equation} 
\min_{\substack{v \in \R^n}}  \int_{\XX}  \varepsilon \log 
\Big(  \frac{1}{n} \sum_{j=1}^{n} \exp \Bigl( \dfrac{ v_j - c(x,Y_j) }{\varepsilon} \Bigr)   \Big) d\mu(x)   -  \frac{1}{n}\sum_{j=1}^{n}  v_j + \varepsilon.  \label{Semi-dualdisc0}
\vspace{-0.2cm}
\end{equation}
However, these approaches are based on a sample $(X_1,\ldots,X_m)$ from $\mu$, to solve the OT problem between {\it the absolutely continuous measure $\mu$ and the discrete measure $ \widehat{\nu}_{n}$} when $m \to + \infty$ and $n$ is held fixed. 
The originality of our approach is to make use of a sample $(Y_1,\ldots,Y_n)$ from $\nu$ to solve the regularized OT problem between {\it two absolutely continuous measures $\mu$ and $\nu$} when $n \to + \infty$. 
In this continuous setting, \cite{aude2016stochastic} also proposed a RKHS parametrization of a pair of dual potentials, which is much different from the Fourier decomposition of a dual potential in the semi-dual formulation as proposed in this paper.

\subsubsection{Comparison with existing works for MK quantiles estimation}\label{comparisonMaps}

The convergence properties of the empirical transport map \eqref{DEFQn0} to estimate the un-regularized MK quantile map \eqref{DEFQ} have been studied in \cite{chernozhukov2015mongekantorovich,ghosal2021multivariate}. Nevertheless, these estimators take their values in the sample $(Y_1,\cdots,Y_n)$, and regularizing is required to interpolate between these observations.
This was done in \cite{beirlant2019centeroutward,Hallin-AOS_2021} based on optimal couplings $(X_n,Y_n)$, inherited from discrete OT. 
The use of Moreau envelopes in \cite{Hallin-AOS_2021} preserves the cyclical monotonicity as well as the couplings $(X_n,Y_n)$. These are ideal theoretical properties, but a supplementary gradient descent is required when computing a single $Q(x)$ for $x\in\XX$.  
This is alleviated in
\cite{beirlant2019centeroutward} with an approximation of $Q$ rather than an interpolation. 
More precisely, given the unregularized solution $v$ of the problem \eqref{Semi-dualdisc0} for $\ee=0$, the authors approximate its $c$-transform $v^c$ by a LogSumExp. This yields a smooth estimator that is cyclically monotone, but based on an un-regularized dual potential $v$. 
In comparison, the use of EOT in our procedure represents a step towards more regularization, with a cyclically monotone estimator 
related to recent advances in computational OT. Note that
EOT was also recently used in the MK quantiles' framework in \cite{Carlier:2022wq,Masud2021}.

\subsection{Organization of the paper}

Our paper is organized as follows. \Cref{sec:Fourier_coeffs} details the formulation of our algorithm in the space of Fourier coefficients. 
The main results about the convergence of our stochastic algorithm are given in \Cref{sec:MainRes}. 
In \Cref{sec:useful},  we state various keystone properties of the objective functions involved in the stochastic formulation of EOT in the space of Fourier coefficients. 
 Then, in \Cref{sec:num}, we illustrate the performances of our new algorithm on simulated data. In particular,  
  the methodology to obtain a map from the \textit{spherical uniform} distribution instead of the uniform  distribution  on the unit hypercube is explained. In these numerical experiments, by letting $\varepsilon$ varying, we also study the effect of the entropic regularization on the estimation of the MK quantile function $Q$. A conclusion and a discussion on some perspectives are given in \Cref{CCL}. All the proofs are postponed to a technical Appendix.  Finally, additional proofs on the differentiability of the objective functions  are given in  supplement materials. 
  
For the sake of reproducible research, the Python codes for the experiments carried out in this paper are available at \href{https://github.com/gauthierthurin/SGD_Space_Fourier_coeffs}{https://github.com/gauthierthurin/SGD\_Space\_Fourier\_coeffs}. 


\section{A new stochastic algorithm in the space of Fourier coefficients} \label{sec:Fourier_coeffs}

\subsection{Our approach}

From now on and throughout the paper, $\mu$ is assumed to be the uniform distribution on $\XX = [0,1]^d$, except in some of the numerical experiments carried out in \Cref{sec:num} where a change of variable enables to consider the spherical uniform distribution for which $\XX = \Bb^{d}$.
Then, we consider the normalization condition for the dual potentials
\begin{equation}
\int_{\XX} u(x) d \mu(x) = 0. \label{eq:normcond}
\end{equation}
Taking the support of $\mu$ to be equal to $[0,1]^d$ is motivated by the choice to parametrize a dual function $u \in L^{1}(\mu)$, satisfying the identifiability condition \eqref{eq:normcond}, by 
its decomposition in the standard Fourier basis $\phi_\lambda(x)=e^{2\pi i \langle \lambda,x \rangle}$, for $\lambda \in \ZZ^d$, that is
\begin{equation*}
u(x) = \sum_{\lambda \in \Lambda} \theta_\lambda \phi_\lambda(x),
\end{equation*}
where $\Lambda = \ZZ^{d} \backslash \{0\}$ and $\theta=(\theta_\lambda)_{\lambda \in \Lambda} $ are the Fourier coefficients of $u$, 
\begin{equation*}
\theta_\lambda = \int_{\XX} \overline{\phi_\lambda(x)} u(x) d\mu(x).
\end{equation*}
We refer to \cite{SteinWeiss} for an introduction to multiple Fourier series on the \textit{flat torus} $\TT^d=\R^d/\ZZ^d$.
Hereafter, $\TT^d$ stands for the set of equivalence classes $[x] = \{ x + k \; ; k\in \ZZ^d\}$ for all $x\in [0,1[^d$. 
With a slight abuse of notation, we identify $\TT^d$ to its fundamental domain $[0,1[^d$, so that integration on $\TT^d$ is Lebesgue-integration on $[0,1[^d$, see \cite{SteinWeiss} or \cite{CORDEROERAUSQUIN1999199,Manole2021} in the OT literature.
Then, for a given regularization parameter $\varepsilon > 0$, we rewrite the dual problem \eqref{eq:dualOTreg} with this parametrization, to consider, for $\ell_{1}(\Lambda)$ defined hereafter, the following {\it stochastic convex minimisation} problem
\begin{equation}
\tht^\ee = \argminE_{\theta \in \ell_{1}(\Lambda)} \; H_\ee(\theta) \hspace{1cm} \mbox{with}  \hspace{1cm}  H_\ee(\theta) = \mathbb{E} \left[ h_\ee(\theta,Y) \right] \label{Se}
\end{equation}
where $Y$ is a random vector with distribution $\nu$ and
\begin{equation*}
h_\ee(\theta,y) = \ee \log \left( \int_{\XX} \exp \left( \frac{\SL \theta_\lambda \phi_\lambda(x) -c(x,y)}{\ee} \right) d\mu(x) \right) + \ee.
\end{equation*}
There, we refer to \cite[Chapter 8]{Jost2003} for a basic course on Fr\'echet differentiability and Taylor formulas for  
functions between Banach spaces.
In \Cref{sec:useful}, it is shown that, for every $y \in \YY$, the function $\theta \mapsto h_\ee(\theta,y) $ is Fr\'echet differentiable only if $\tht$ belongs to the convex set
\begin{equation*}
\ell_{1}(\Lambda)  = \left\{\theta = (\theta_\lambda)_{\lambda \in \Lambda} \in \CC^{\Lambda} \; : \; \theta_{-\lambda} = \overline{\theta_\lambda} \mbox{ and } \| \theta \|_{\ell_1} = \sum_{\lambda \in \Lambda} |\theta_\lambda| < + \infty    \right \}.
\end{equation*}
Moreover, its differential $D_\theta h_\ee(\theta, y)$ is identified as an element of the dual Banach space 
\begin{equation*}
\ell_{\infty}(\Lambda)  = \left\{v = (v_\lambda)_{\lambda \in \Lambda} \in \CC^{\Lambda} : v_{-\lambda} = \overline{v_\lambda} \mbox{ and } \| v \|_{\ell_{\infty}} = \sup_{\lambda \in \Lambda} |v_\lambda| < + \infty    \right\}.
\end{equation*}
 The components of the first order Fr\'echet derivative  $D_\theta h_\ee(\theta, y)$ are the partial derivatives   
\begin{equation}
\label{partialderh}
\frac{\partial h_\ee(\theta,y) }{\partial \theta_\lambda} = \int_\XX \overline{\phi_\lambda(x)} F_{\theta,y}(x) d\mu(x)
\end{equation}
that are the Fourier coefficients of the function
\begin{equation}\label{def_F_tht}
F_{\theta,y}(x) = \frac{\exp \left( \frac{\SL \theta_\lambda \phi_\lambda(x) -c(x,y)}{\ee} \right) }{ \int_\XX \exp \left( \frac{\SL \theta_\lambda \phi_\lambda(x) -c(x,y)}{\ee} \right) d\mu(x)}.
\end{equation}
One can observe that $F_{\theta,y}$ is a probability density function, which is a key property that we shall repeatedly use. 
In this paper, we shall analyze \eqref{Se} as a stochastic  convex minimisation problem over the Banach space $(\ell_{1}(\Lambda), \| \cdot \|_{\ell_1})$, that corresponds to the formulation of a regularized dual problem of OT in the space of Fourier coefficients.  


Imposing that the Fourier coefficients $\theta=(\theta_\lambda)_{\lambda \in \Lambda} $ form an absolutely convergent series implicitly requires that the  optimal dual potential minimizing  \eqref{eq:dualOTreg} satisfy periodic conditions at the boundary of $[0,1]^d$. 
For readability of the paper, a detailed discussion on sufficient conditions for the  un-regularized optimal dual potential $u_0$  to be periodic  is postponed to \Cref{sec:periodic} in the supplementary material.

Let $(Y_n)$ be a sequence of independent random vectors sharing the same distribution  $\nu$. In the spirit of
\cite{robbins_monro}, we propose to estimate the solution of \eqref{Se} by considering the stochastic algorithm in the Banach space $(\ell_{1}(\Lambda), \| \cdot \|_{\ell_1})$ defined, for all
$n \geq 0$, by
\begin{equation}
\widehat{\theta}_{n+1} = \widehat{\theta}_{n} - \gamma_n W D_\theta h_\ee(\widehat{\theta}_{n} , Y_{n+1})
\label{defthetan}
\end{equation}
where $\gamma_n = \gamma n^{-c}$ with $\gamma > 0$ and $1/2 < c \leq 1$, which clearly implies the standard conditions
\begin{equation}
\label{condgamma}
\sum_{n = 0}^\infty \gamma_n  = + \infty \hspace{1cm} \mbox{and} \hspace{1cm} \sum_{n = 0}^\infty \gamma_n^2  < + \infty.
\end{equation}
Moreover, $W$ is the following linear operator
\begin{equation*}
\left\{\begin{array}{ccc}
W :  (\ell_{\infty}(\Lambda), \| \cdot \|_{\ell_{\infty}})  & \to &  (\ell_{1}(\Lambda), \| \cdot \|_{\ell_1}) \\
v = (v_\lambda)_{\lambda \in \Lambda}  & \mapsto & w \odot v = (w_\lambda v_\lambda)_{\lambda \in \Lambda}
\end{array}\right.
\end{equation*}
where $w = (w_\lambda)_{\lambda \in \Lambda}$ is a {\it deterministic sequence of positive weights} satisfying the normalizing condition
\begin{equation}
\label{condw}
\|w\|_{\ell_1} = \sum_{\lambda \in \Lambda} w_\lambda < + \infty.
\end{equation}

A main difficulty arising here is that the space $\ell_{1}(\Lambda)$ of parameters differs from its dual space $\ell_{\infty}(\Lambda)$ to which the Fr\'echet derivative $D_\theta h_\ee(\theta, y)$ belongs. This is a classical issue when considering convex optimization in Banach spaces, see e.g.\ \cite{Bubeck}, and this is the reason why we introduce the linear operator $W$ in  \eqref{defthetan} that maps  $\ell_{\infty}(\Lambda)$ 
to $\ell_{1}(\Lambda)$. The use of the linear operator $W$ also induces two weighted norms on the space
\begin{equation*}
\ell_{2}(\Lambda)  = \left\{\theta = (\theta_\lambda)_{\lambda \in \Lambda} \in \CC^{\Lambda} : \theta_{-\lambda} = \overline{\theta_\lambda} \mbox{ and } \| \theta \|_{\ell_2}^{2} = \sum_{\lambda \in \Lambda} |\theta_\lambda|^2 < + \infty    \right \}.
\end{equation*}
One can observe that  we clearly have $\ell_{1}(\Lambda) \subset \ell_{2}(\Lambda)$.

\begin{definition}\label{def:normsW}
For every $\tht \in \ell^2(\Lambda)$ and for a sequence $w =  (w_\lambda)_{\lambda \in \Lambda}$ of positive weights satisying \eqref{condw}, we define the two weighted norms
\begin{equation}
\label{DEFnormW}
\Vert \tht \Vert^2_W=  \SL w_\lambda \vert  \tht_\lambda \vert^2  \hspace{1cm} \text{ and  } \hspace{1cm}
\Vert \tht \Vert^2_{W^{-1}} = \SL w_\lambda^{-1} \vert  \tht_\lambda \vert^2.
\end{equation}
\end{definition}

The aim of this paper is to establish consistency results for the stochastic algorithm given by \eqref{defthetan}.
Hereafter, 
a {\it regularized estimator} of the optimal potential 
defined, for $x \in \XX$, by
\begin{equation}
u_\ee(x) = \SL \theta_\lambda^{\ee} \phi_\lambda(x) 
\label{DEFu_e}
\end{equation}
is naturally given by 
\begin{equation}
\label{DEFun}
\widehat{u}_{\varepsilon}^{\, n}(x) = \SL \widehat{\theta}_{n,\lambda} \phi_\lambda(x).
\end{equation}

\noindent In practice, our numerical procedure starts by considering a discretization of the dual potential $u$ over a regular grid $\XX_p = \{x_1,\ldots,x_p\}$ of points in $\XX$. 
This allows us to compute the corresponding set of Fourier coefficients at frequencies $\Lambda_p$ of size $p$ by the Fast Fourier Transform (FFT). 
Then, the sequence $(\widehat{\theta}_{n,\lambda})_{\lambda \in \Lambda_p}$ satisfying \eqref{defthetan} is easily implemented using, at each iteration, the FFT and its inverse, see \Cref{alg}  below. 
Hence, the computational cost, at each iteration, of our algorithm  is of order $\bigO\left( p \log(p)\right)$, while the cost of the celebrated Sinkhorn algorithm \cite{cuturi2013sinkhorn} is $\bigO\left( pn \right)$, using a discrete source measure supported on $\XX_p$, and the one of the stochastic algorithms proposed in \cite{bercu2020asymptotic,BB_GN,aude2016stochastic} is $\bigO\left( n \right)$ at each iteration.

In our approach, the computational cost depends on the size $p$ of the  grid on $\XX_p$ 
that is fixed by the user. 
This size $p$ does not require to be particularly large, as showed by numerical experiments.
However, we stress that this appealing computational cost of $\bigO\left( p \log(p)\right)$ comes with a drawback regarding the dimension.
Indeed, the number $p$ of points in a uniform grid on $[0,1]^d$ grows exponentially with $d$. 
Thus, a standard implementation of the FFT on a uniform grid becomes difficult for medium dimensions such as $d=10$.  
Extending our work to the high-dimensional setting would require the study of more sophisticated FFTs as proposed in \cite{potts2021approximation}, but this issue is beyond the scope of this paper.

\begin{algorithm}
\caption{Stochastic algorithm \eqref{defthetan}}\label{alg}
\begin{algorithmic}
\State Initialize $N\in \mathbb{N}$, $\XX_p = \{x_1,\cdots,x_p\}$, $u \in \R^p$ and $W \in \R^{p\times p}$
\State $\tht \gets \mbox{FFT}(u)$
\While{$n \leq N$}
\State $y \gets Y_{n}$
\State $u \gets \mbox{IFFT}(\tht)$
\For{$i \in \{1,\cdots,p\}$}
\State $F[i] \gets \exp\Big((u[i] - c(x_i,y))/\ee\Big)$
\EndFor
\State $F \gets F/\mbox{mean}(F)$ \Comment{estimate of \eqref{def_F_tht}}
\State $\mbox{grad} \gets \mbox{FFT}(F)$
\State $\tht \gets \tht -\gamma_n W \cdot \mbox{grad} $
\EndWhile
\end{algorithmic}
\end{algorithm}

\subsection{The barycentric projection}\label{sec:barycproj}

Inspired by \eqref{DEFQgrad}, we could propose to estimate the MK quantile function via the {\it regularized estimator}  $\widehat{Q}_{\varepsilon}^n(x) = x -  \nabla \widehat{u}_{\varepsilon}^n(x).$ 
However, $\widehat{u}_{\varepsilon}^n$ is not necessarily a concave function, and thus $\widehat{Q}_{\varepsilon}^n$ does not correspond to the gradient of a convex function, that is the desired multivariate monotonicity for a quantile function, as argued in \cite{Hallin-AOS_2021}. To the contrary,
the \textit{entropic map} studied in \cite{pooladian2021entropic} is the gradient of a convex function as shown in \cite{chewi2022entropic}[Lemma 1]. 
Since the entropic map can be estimated from any solution of the EOT problem \eqref{eq:dualOTreg}, we propose in this paper the following estimator derived from \eqref{DEFun},
\begin{equation}
\widehat{Q}^n_{\varepsilon}(x)   = \sum_{j=1}^{n} \widehat{F}_{j}(x) Y_{j} 
\hspace{1cm} \mbox{where} \hspace{1cm}
\widehat{F}_{j}(x) = \frac{ \exp \Bigl( \dfrac{ (\widehat{u}_{\varepsilon}^n)^{c,\varepsilon}(Y_j) - c(x,Y_j) }{\varepsilon} \Bigr) }{\sum_{\ell = 1}^{n} \exp \Bigl( \dfrac{ (\widehat{u}_{\varepsilon}^n)^{c,\varepsilon}(Y_\ell)  - c(x,Y_\ell) }{\varepsilon} \Bigr)},  \label{eq:hatQ}
\end{equation}
that is obtained by computing the smooth conjugate $(\widehat{u}_{\varepsilon}^n)^{c,\varepsilon}  \in \R^n $  of $\widehat{u}_{\varepsilon}^n$. Note that if one denotes by $((\widehat{u}_{\varepsilon}^n)^{c,\varepsilon})^{c,\varepsilon}(x)$ the  smooth conjugate of $(\widehat{u}_{\varepsilon}^n)^{c,\varepsilon}$ at $x$, then our estimator can also be expressed as
\begin{equation*}
\widehat{Q}^n_{\varepsilon}(x) = x  - \nabla ((\widehat{u}_{\varepsilon}^n)^{c,\varepsilon})^{c,\varepsilon}(x).
\end{equation*}
 Recall that an alternative algorithm to solve the semi-discrete EOT problem is to consider the formulation \eqref{Semi-dualdisc0} as studied in \cite{bercu2020asymptotic,BB_GN,aude2016stochastic}. 
Based on independent samples $X_1,\ldots,X_m$ from $\mu$, these works approach the unique solution $\widetilde{v}_{n} \in \R^n$ of the problem \eqref{Semi-dualdisc0} when $m \to + \infty$ and $n$ is held fixed. 
Then, 
one can estimate the entropic map using, for all $x\in\XX$, 
\begin{equation}
\widetilde{Q}^n_{\varepsilon}(x)   = \sum_{j=1}^{n} \widetilde{F}_{j}(x) Y_{j} 
\hspace{1cm} \mbox{where} \hspace{1cm}
 \widetilde{F}_{j}(x) = \frac{ \exp \Bigl( \dfrac{ \widetilde{v}_{n,j} - c(x,Y_j) }{\varepsilon} \Bigr) }{\sum_{\ell = 1}^{n} \exp \Bigl( \dfrac{ \widetilde{v}_{n,\ell} - c(x,Y_\ell) }{\varepsilon} \Bigr)}.  \label{eq:hattildeQ}
\end{equation}
 The numerical performances of $\widehat{Q}^n_{\varepsilon}(x)$ are compared to those of $\widetilde{Q}^n_{\varepsilon}$ in \Cref{sec:num}.



\section{Main results}\label{sec:MainRes}

In order to state our main results, it is necessary to introduce two suitable assumptions related to 
the optimal sequence of Fourier coefficients $\theta^{\varepsilon} =  (\theta_\lambda^{\varepsilon})_{\lambda \in \Lambda}$ and
the second order Fr\'echet derivative of the function $H_{\varepsilon}$ given by \eqref{Se}.

\begin{hyp} \label{hyp:Winv}
The  sequence of Fourier coefficients $(\theta_\lambda^{\varepsilon})_{\lambda \in \Lambda}$ satisfies
$
\Vert \tht^\ee \Vert_{W^{-1}} < + \infty.
$
\end{hyp}

\begin{hyp} \label{hyp:DH2OPT}
For any regularization parameter $\varepsilon > 0$, there exists a positive constant $c_\ee$ such that 
the second order Fr\'echet derivative of the function $H_{\varepsilon}$ evaluated at the optimal value $\theta^{\varepsilon}$
satisfies, for any  $ \tau \in  \ell_{1}(\Lambda)$,
\begin{equation}
\label{DH2OPT}
D^2 H_\ee(\theta^{\varepsilon})[\tau,\tau]\geq c_\ee \Vert \tau \Vert^2_{\ell^2}.
\end{equation}
\end{hyp}

Our main theoretical result is devoted to the almost sure convergence of the random  sequence $(\widehat\theta_n)_n$ defined by \eqref{defthetan}.

\begin{thm}
\label{thmAS} 
Suppose that the initial value $\widehat\theta_0$ is any random element in $\ell^2(\Lambda)$ such that $\Vert  \widehat\theta_0\Vert_{W^{-1}}  < +\infty$.
Then, under Assumptions \ref{hyp:Winv} and \ref{hyp:DH2OPT}, the sequence $(\widehat\theta_n)$  converges almost surely in $\ell_2$ towards the solution $\theta^{\varepsilon}$
of the stochastic convex minimisation problem \eqref{Se}, i.e.
\begin{equation}
\label{thmAS1}
\lim_{n \to \infty}\Vert \widehat\theta_n - \theta^{\ee}   \Vert_{\ell_2} = 0 \hspace{1cm} \text{a.s.}
\end{equation}
Equivalently, we also have that
\begin{equation}
\label{thmAS2}
\lim_{n \to \infty}  \int_{\XX}|\widehat{u}_{\varepsilon}^{\, n}(x)  - u_\ee(x)|^2 d\mu(x) =  0 \hspace{1cm} \mbox{a.s.}
\end{equation}
\end{thm}

\Cref{hyp:Winv} can be made more explicit by the choice of a specific sequence of weights $w = (w_\lambda)_{\lambda \in \Lambda}$ and by imposing regularity assumptions on the function $u_\ee \in L^1(\mu)$ given by \eqref{DEFu_e}.
For example, one may assume in dimension $d=2$ that $u_\ee$ is differentiable (with periodic conditions on the boundary on $\XX$) and that its gradient is square integrable,
\begin{equation*}
\int_{\XX} \| \nabla u(x) \|^2 d\mu(x) < + \infty.IU
\end{equation*}
Then, under such assumptions, one may use the fact that $\nabla u(x) =  \SL 2 \pi i \lambda \theta_\lambda^{\ee}  \phi_\lambda(x) $ and Parseval's identity, \cite{SteinWeiss}[Theorem 1.7], to obtain that 
\begin{equation*}
\SL \Vert \lambda \Vert^2 \vert \theta_\lambda^\ee \vert^2 < + \infty.
\end{equation*}
Consequently, for the specific choice $w_\lambda = \| \lambda \|^{-2}$, we find that \Cref{hyp:Winv} holds properly. In higher dimension $d$, it is necessary to make additional assumptions on the differentiability of $u_\ee$.
Note that we shall also prove in \Cref{borneINF_hessienne}  that for any $\theta,\tau \in \overline{\ell_1}(\Lambda)$, 
\begin{equation*}
D^2 H_\ee(\theta^\ee)[\tau,\tau] \ge \frac{1}{\ee} \left(2 - \int_\YY \int_\XX F_{\tht^\ee,y}^2(x) d\mu(x) d\nu(y)\right)\Vert \tau \Vert_{\ell_2}^2.
\end{equation*}
Therefore a sufficient condition for \Cref{hyp:DH2OPT} to hold is to assume that 
\begin{equation*}
\int_\YY \int_\XX F_{\tht^\ee,y}^2(x) d\mu(x) d\nu(y) < 2
\hspace{0.5cm} \mbox{with} \hspace{0.5cm}
c_\ee =  \frac{1}{\ee} \left(2 - \int_\YY \int_\XX F_{\tht^\ee,y}^2(x) d\mu(x) d\nu(y)\right).
\end{equation*}

\section{Properties of the objective  function $H_{\varepsilon}$} \label{sec:useful}

The purpose of this  section is to discuss various keystone properties of the functions $h_\ee$ and 
$H_{\varepsilon}$
that are needed to establish our main result on the convergence of our stochastic algorithm $\widehat{\theta}_{n}$.

Throughout this section, it is assumed that $\varepsilon > 0$. Moreover, all the results stated below are valid for any cost function $c$ that is lower semi-continuous and that belongs  to $L^1(\mu \otimes \nu)$ so that  regularized OT is well defined. Consequently, the restriction to the quadratic cost is no longer needed in this section.

Let us first discuss the first and second order Fr\'echet differentiability of the functions $H_{\varepsilon}$ and $h_\ee$ that are functions from the  Banach space $(\bar{\ell}_{1}(\Lambda), \|\cdot\|_{\ell_1})$ to $\R$.  The following proposition gives the expression of the first order Fr\'echet derivative, that we shall sometimes refer to as the gradient, of $h_{\varepsilon}$ and $H_\ee$, as well as upper bounds on their operator norm.

\begin{prop}\label{prop:grad}
For any $y \in \YY$, the first order Fr\'echet derivative of the function $h_\ee(\cdot,y)$  at $\theta \in \bar{\ell}_{1}(\Lambda)$  is the  linear operator $D_{\theta} h_\ee(\theta, y) :  \bar{\ell}_{1}(\Lambda) \to \R$ defined for any  $ \tau \in  \bar{\ell}_{1}(\Lambda)$ as 
\begin{equation}\label{Gradh}
D_{\theta} h_\ee(\theta, y)[ \tau] = \SL   \overline{\frac{\partial h_\ee(\theta,y) }{\partial \theta_\lambda}}  \tau_{\lambda}
\end{equation}
where
\begin{equation}
\label{part_deriv_h_ee}
\frac{\partial h_\ee(\theta,y) }{\partial \theta_\lambda} = \int_\XX \overline{\phi_\lambda(x)} F_{\theta,y}(x) d\mu(x).
\end{equation}
Moreover, the linear operator $D_{\theta} h_\ee(\theta, y)$ can be identified as an element of $\bar{\ell}_{\infty}(\Lambda)$  and its operator norm satisfies, for any $\theta \in \bar{\ell}_{1}(\Lambda)$ and $y \in \YY$,
\begin{equation} \label{eq:opnormgradienth}
\| D_{\theta} h_\ee(\theta, y)  \|_{op} = \sup_{\| \tau \|_{\ell_1} \leq 1}  |D_{\theta} h_\ee(\theta, y)[ \tau]| \le \sup\limits_{\lambda \in \Lambda} \left\vert \frac{\partial h_\ee(\theta,y) }{\partial \theta_\lambda} \right\vert  \leq 1.
\end{equation}
The first order Fr\'echet derivative of the function  $H_{\varepsilon}$ at $\theta \in \bar{\ell}_{1}(\Lambda)$  is the  linear operator $D H_{\varepsilon}(\theta) :  \bar{\ell}_{1}(\Lambda) \to \R$ defined for any  $ \tau \in  \bar{\ell}_{1}(\Lambda)$ as 
\begin{equation}\label{GradH}
D H_{\varepsilon}(\theta)[ \tau] = \SL  \overline{\frac{\partial H_{\varepsilon}(\theta)}{\partial \theta_\lambda}}   \tau_{\lambda}
\end{equation}
where
\begin{equation*}
\frac{\partial H_{\varepsilon}(\theta)}{\partial \theta_\lambda}   =  \int_{\YY} \frac{\partial h_\ee(\theta,y) }{\partial \theta_\lambda}   d \nu(y).
\end{equation*}
Moreover, the operator norm of the  linear operator 
$
D H_{\varepsilon}(\theta) 
$
satisfies, for any $\theta \in \bar{\ell}_{1}(\Lambda)$,
\begin{equation} \label{eq:opnormgradientH}
\| D H_{\varepsilon}(\theta)  \|_{op} = \sup_{\| \tau \|_{\ell_1} \leq 1}  |D H_{\varepsilon}(\theta) [ \tau]| \leq 
\sup\limits_{\lambda \in \Lambda} \left\vert \frac{\partial H_\ee(\theta) }{\partial \theta_\lambda} \right\vert \leq 1.
\end{equation}
\end{prop}

The proposition below gives the expression of the second order Fr\'echet derivative, that we shall sometimes refer to as the Hessian, 
of $h_{\varepsilon}$ and $H_\ee$ and upper bounds on their operator norm.

\begin{prop}\label{prop:Hess}
For any $y \in \YY$, the second order Fr\'echet derivative of the function $h_\ee(\cdot,y)$  at $\theta \in \bar{\ell}_{1}(\Lambda)$  is the following symmetric  bilinear mapping  from $\bar{\ell}_{1}(\Lambda)  \times \bar{\ell}_{1}(\Lambda) $ to $\R$
\begin{eqnarray}\label{eq:D2h_ee}
\hspace{1cm} D^2_\tht h_\ee(\theta,y)[\tau,\tau']  & = &  \frac{1}{\ee}  \sum\limits_{\lambda' \in \Lambda} \SL \tau'_{\lambda'} \overline{\tau_{\lambda}} \int_{\XX }\phi_{\lambda'}(x) \overline{\phi_{\lambda}(x)} F_{\theta,y}(x) d\mu (x)\\
 & -&  \frac{1}{\ee} \left( \SL \tau'_{\lambda} \int_{\XX } \phi_{\lambda}(x) F_{\theta,y}(x) d\mu(x)  \right) \!\!\overline{\left( \SL \tau_{\lambda} \int_{\XX } \phi_{\lambda(x)}F_{\theta,y}(x) d\mu(x) \right)} \nonumber.
\end{eqnarray}
and its operator norm satisfies, for any $\theta \in \bar{\ell}_{1}(\Lambda)$ and $y \in \YY$,
\begin{equation} \label{eq:opnormHessh}
\|D^2_\tht h_\ee(\theta,y)\|_{op} = \sup_{\| \tau \|_{\ell_1} \leq 1, \| \tau' \|_{\ell_1} \leq 1} |D^2_\tht h_\ee(\theta,y) [\tau,\tau'] | \leq  \frac{1}{\ee}.
\end{equation}
Moreover, the second order Fr\'echet derivative  of $H_{\varepsilon} : \bar{\ell}_{1}(\Lambda) \to \R$ is the symmetric  bilinear mapping  from $\bar{\ell}_{1}(\Lambda)  \times \bar{\ell}_{1}(\Lambda) $ to $\R$ defined by
\begin{eqnarray}
\label{eq:D2H}
\hspace{1cm}  D^2 H_\ee(\theta)[\tau,\tau']  & = &  \frac{1}{\ee}  \sum\limits_{\lambda' \in \Lambda} \SL \tau'_{\lambda'} \overline{\tau_{\lambda}}\int_{\YY }\!\! \int_{\XX }\phi_{\lambda'}(x) \overline{\phi_{\lambda}(x)} F_{\theta,y}(x) d\mu (x) d\nu(y)  \\
&-&  \!\! \frac{1}{\ee} \int_{\YY } \!\!\left( \SL \tau'_{\lambda} \! \int_{\XX } \!\!\phi_{\lambda}(x) F_{\theta,y}(x) d\mu(x)  \right)\!\! \overline{\left( \SL \tau_{\lambda}  \!\int_{\XX } \!\!\phi_{\lambda}(x)F_{\theta,y}(x) d\mu (x)\right)} d\nu(y), \nonumber
\end{eqnarray}
and its operator norm satisfies, for any $\theta \in \bar{\ell}_{1}(\Lambda)$,
\begin{equation} \label{eq:opnormHessH}
\|D^2 H_\ee(\theta)\|_{op} = \sup_{\| \tau \|_{\ell_1} \leq 1, \| \tau' \|_{\ell_1} \leq 1} |D^2 H_\ee(\theta) [\tau,\tau'] | \leq  \frac{1}{\ee}.
\end{equation}
\end{prop}
%
We now provide useful results on the regularity of $H_\ee$.
\begin{prop}\label{prop:convexity}
For any $y \in \YY$,   the functions $h_{\ee}(\cdot,y)$ and $H_{\varepsilon}$ are strictly convex on $\bar{\ell}_{1}(\Lambda)$. 
\end{prop}

As already noticed in previous works \cite{bercu2020asymptotic,aude2016stochastic} dealing with related objective functions, the function $H_\ee$ is not strongly convex. 
Nevertheless, one can obtain a local strong convexity property of the function $H_\ee$ in the neighborhood of its minimizer $\theta^{\varepsilon}$. This result is a consequence of the notion of generalized self-concordance introduced in \cite{Bach14},
which has been
shown to hold for regularized semi-discrete OT in \cite{bercu2020asymptotic}, and which we extend to the setting of the functional $H_\ee$ on the Banach space $ \bar{\ell}_{1}(\Lambda)$.

\begin{prop}\label{prop:reg_Hee}
For all $\theta \in \bar{\ell}_{1}(\Lambda)$, we have
\begin{equation}\label{eq:boundHeps}
 H_\ee(\theta) - H_\ee(\tht^\ee) \le \frac{1}{\ee}\Vert \tht - \theta^\ee\Vert_{\ell_1}^2.
\end{equation}
Moreover, for any $\theta \in \overline{\ell_1}(\Lambda)$, the following local strong convexity property holds 
\begin{equation}\label{eq:Grad_vs_hess}
DH_\ee(\tht)[\tht-\tht^\ee] \ge 
g\Big(\frac{2}{\ee}  \Vert \tht - \tht^\ee\Vert_{\ell_1}\Big) D^2H_\ee(\tht^\ee)[\tht-\tht^\ee,\tht-\tht^\ee],
\end{equation}
where,
for all $x>0$,
\begin{equation}\label{defg}
g(x)=\frac{1-\exp(-x)}{x}.
\end{equation}

\end{prop}

\section{Numerical experiments} \label{sec:num} 

\subsection{Influence of the dimension $d$}
We first investigate the convergence of our numerical scheme for the estimation of the entropic map using various values of the dimension $d$ to analyse its impact of the computational performances of our approach. 

To do so, our estimator $\widehat{Q}^n_{\varepsilon}(x)$ in \eqref{eq:hatQ} is compared to $\widetilde{Q}^n_{\varepsilon}(x)$ in \eqref{eq:hattildeQ} where the dual potential $\widetilde{v}_n \in \R^n$ needed to compute $\widetilde{Q}^n_{\varepsilon}(x)$ is obtained  with either the Sinkhorn algorithm \cite{cuturi2013sinkhorn} or a stochastic algorithm as proposed in \cite{bercu2020asymptotic,aude2016stochastic}. 
Starting from the uniform distribution on $[0,1]^d$, we consider the map 
$
Q : x \mapsto L^TL x + b
$ 
where $L$ is a lower triangular matrix and $b\in \R^d$, both filled with ones.
Trivially, $Q$ is the gradient of a convex function, so that it is the MK quantile function of $\nu = Q_\# \mu$. 
Thus, by Monte-Carlo sampling, we are able to approximate the mean squared error of any estimator  $\widehat{Q}$ defined as
\begin{equation}\label{defMSE}
\mbox{MSE}(\widehat{Q} ) = \mathbb{E}\left[ \Vert \widehat{Q} (X) - Q(X) \Vert^2\right].
\end{equation}
The three ways of estimating $Q$ are based on iterative schemes that we let running until convergence of the MSE below the value $10^{-2}$ for $d=2,3,4$, and by taking $\ee=0.005$.

\Cref{fig:TimeVsOthers} illustrates the time before convergence, in seconds, as a function of $n$ (the number of observations). In what follows,
the continuous, semi-discrete, and discrete approaches  refer to \Cref{alg} with $W$ the identity matrix, the stochastic algorithm from \cite{bercu2020asymptotic,aude2016stochastic}, and the Sinkhorn algorithm \cite{cuturi2013sinkhorn} respectively. 
For $\XX_p$ given in \Cref{alg}, the uniform distribution on $\XX_p$ is taken as a discrete reference measure for the Sinkhorn algorithm to ensure a fair comparison with our algorithm. 
The MSE is estimated through $m=500$ other random samples from $\mu$. 
The size $p$ of the grid $\XX_p$ is maintained comparable in every considered dimensions. Results are averaged over $10$ experiments for several samples $(Y_1,\cdots,Y_n)$, and standard deviation is indicated around each MSE curve.
Overall, these numerical experiments reveal a potentially faster convergence for approaches based on stochastic algorithms when the number of  observations grows. Moreover, our continuous approach slightly outperforms the semi-discrete one in term of computational performances.

\begin{figure}[htbp]
\centering
{\subfigure[$d=2$, $p=20^2$.]{\includegraphics[width=0.32 \textwidth,height=0.32\textwidth]{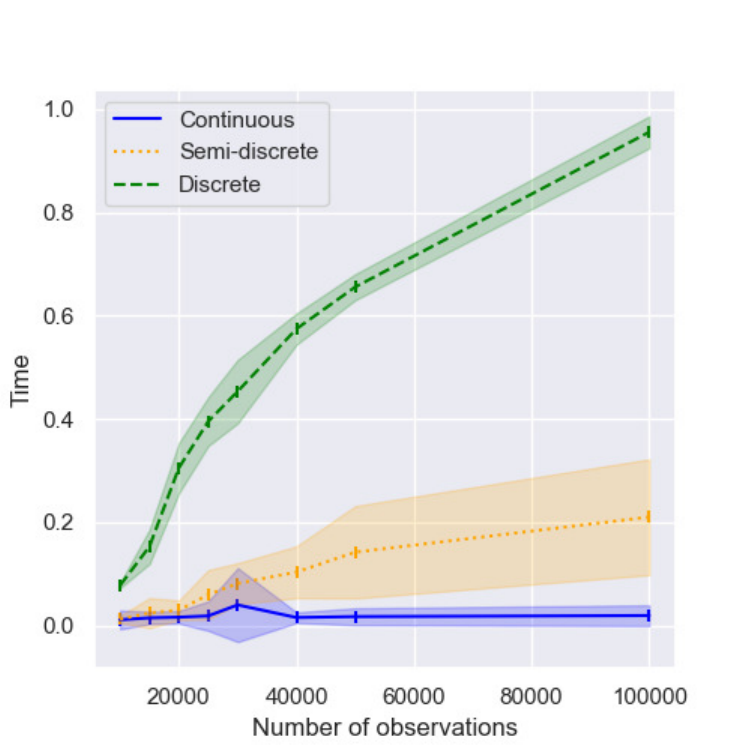}}}
{\subfigure[$d=3$, $p=10^3$.]{\includegraphics[width=0.32 \textwidth,height=0.32\textwidth]{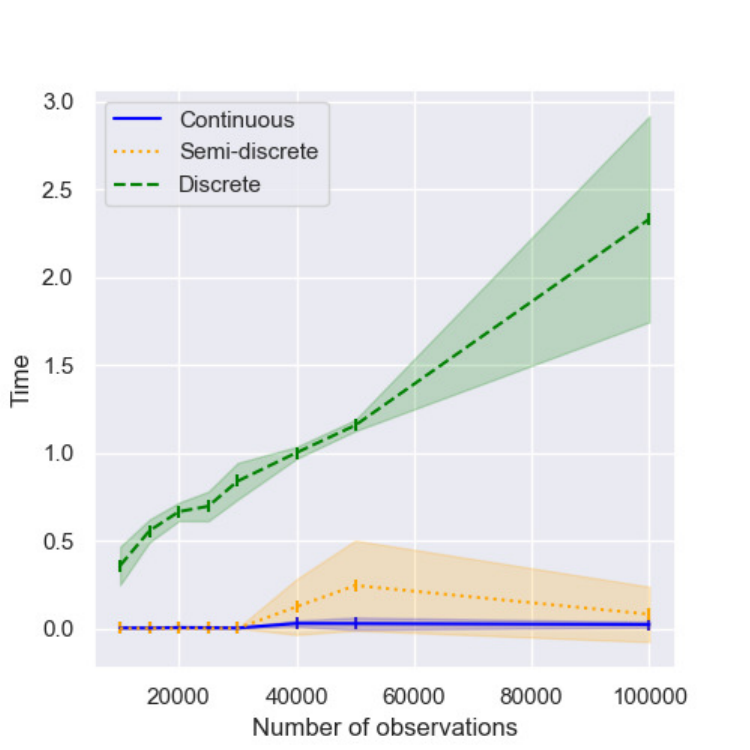}}}
{\subfigure[$d=4$, $p=6^4$.]{\includegraphics[width=0.32 \textwidth,height=0.32\textwidth]{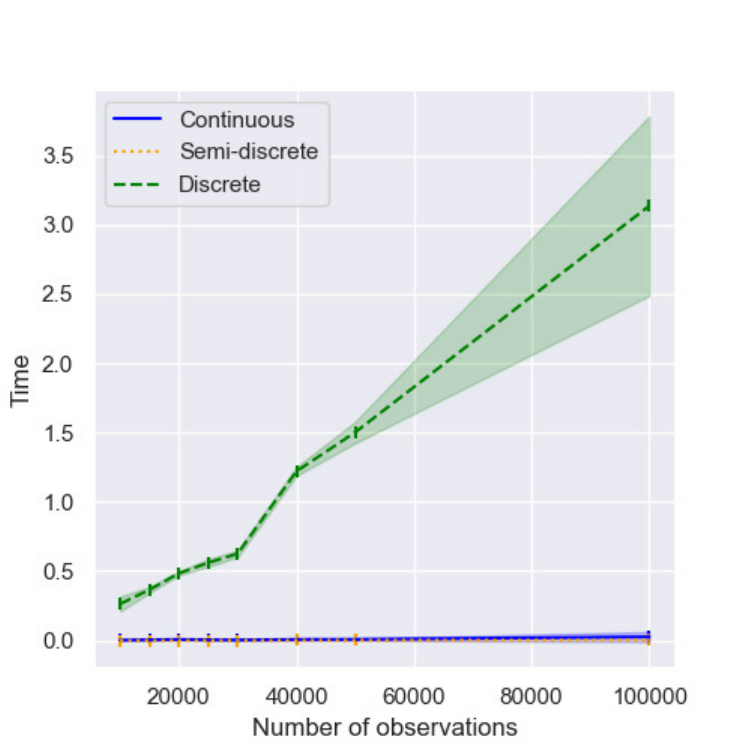}}}

\caption{  Overall time, in seconds, until convergence of the MSE below $10^{-2}$ for different solvers for EOT. \label{fig:TimeVsOthers}}
\end{figure}

 \subsection{Numerical experiments in dimension $d=1$}
The univariate setting allows us an explicit knowledge of the ground truth $Q$. 
There, we study our algorithm with either the standard quadratic cost  in $\R^d$ given by $c(x,y) = \frac{1}{2} \|x-y\|^2$ or the quadratic cost on the flat torus $\TT^d = \R^d/\ZZ^d$ that is
\begin{equation}\label{quad_cost_torus}
c(x,y) = \frac{1}{2} d_{\TT^d}(x,y), \quad \mbox{with}  \quad d_{\TT^d}(x,y) =  \min_{\lambda \in \ZZ^d} \|x-y + \lambda \|.
\end{equation}
The choice of the quadratic cost on the torus is motivated by the discussion in the supplementary material \Cref{sec:periodic} on sufficient conditions related to the  summability of the Fourier coefficients of an optimal dual potential.

For the learning rate $\gamma_n = \gamma n^{-c}$, we took $\gamma = \ee$ and $c = 3/4$. 
The sequence of weights $w = (w_\lambda)_{\lambda \in \Lambda}$ is chosen as $w_{\lambda} = |\lambda|^{-2}$ for $\lambda \in \ZZ \backslash\{0\}$. 
Taking a larger exposant than $2$ results in smoother estimators of the optimal dual potential $u_\ee$. 
For various values of $\ee \in [0.005,0.5]$, we consider a $\mbox{beta}(a,b)$ distribution $\nu$ on $\YY =[0,1]$ with parameters $a=5$ and $b=5$. 
The optimal dual potential $u_0$ and quantile function $Q_0$ are straightforward to compute when $d=1$ for the standard quadratic cost. For a sample of size $n=10^5$, $\widehat{u}_{\varepsilon}^{\, n}$  and $\widehat{Q}_{\varepsilon}^n$ are displayed in  \Cref{fig:1d_beta_distribution}
using either the standard quadratic cost or the  quadratic cost of the torus.
One can observe that the choice of the cost yields a different regularization effect. 
Choosing $\ee = 0.005$ yields values of $\widehat{u}_{\varepsilon}^{\, n}$  and $\widehat{Q}_{\varepsilon}^n$ that are very close to $u_0$ and $Q_0$ respectively.

\begin{figure}[htbp]
\centering
{\subfigure[\mbox{Standard quadratic cost}]{\includegraphics[width=0.45 \textwidth,height=0.3\textwidth]{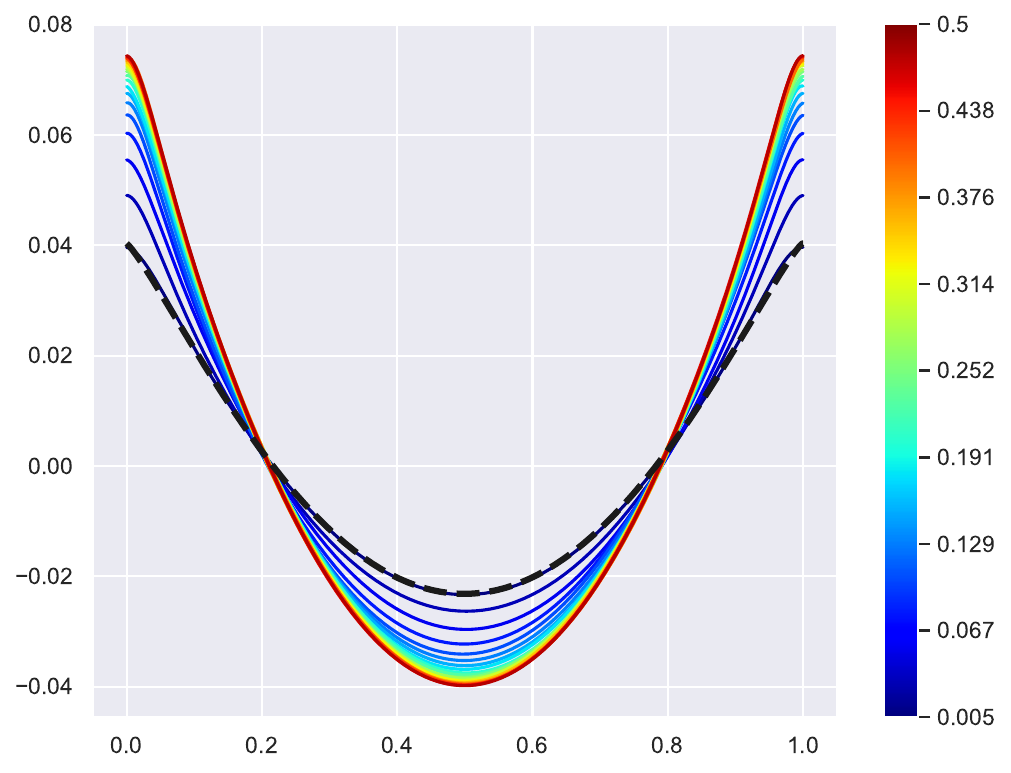}}}
{\subfigure[\mbox{Quadratic cost on the torus}]{\includegraphics[width=0.45 \textwidth,height=0.3\textwidth]{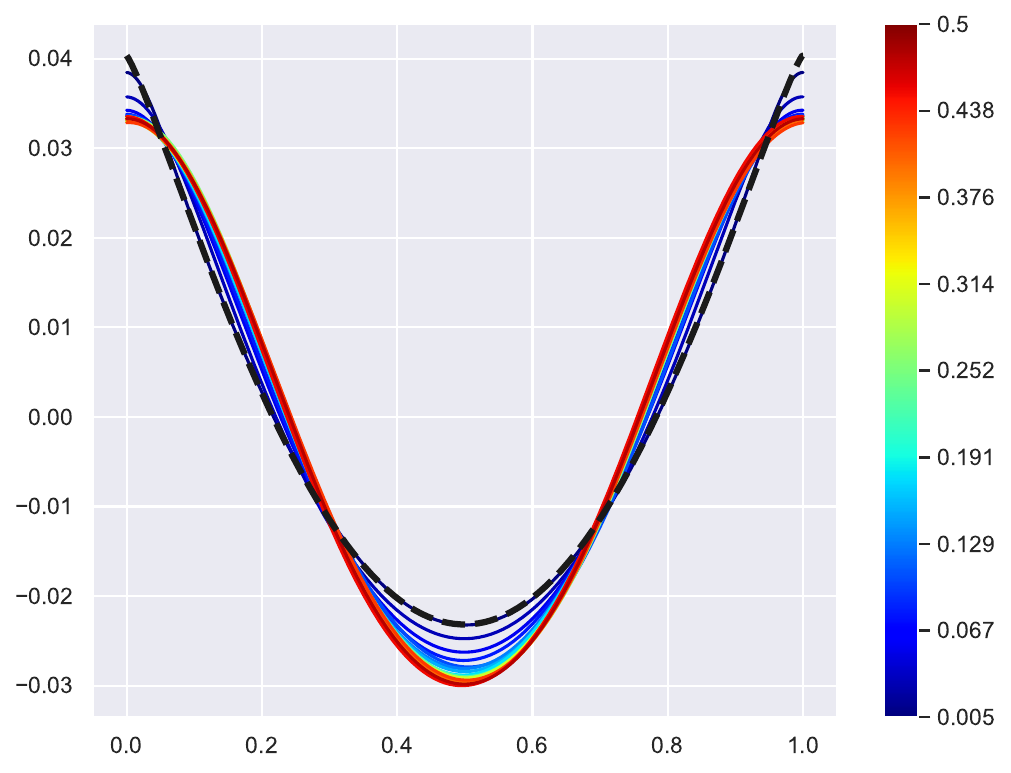}}}

{\subfigure[\mbox{Standard quadratic cost}]{\includegraphics[width=0.45 \textwidth,height=0.3\textwidth]{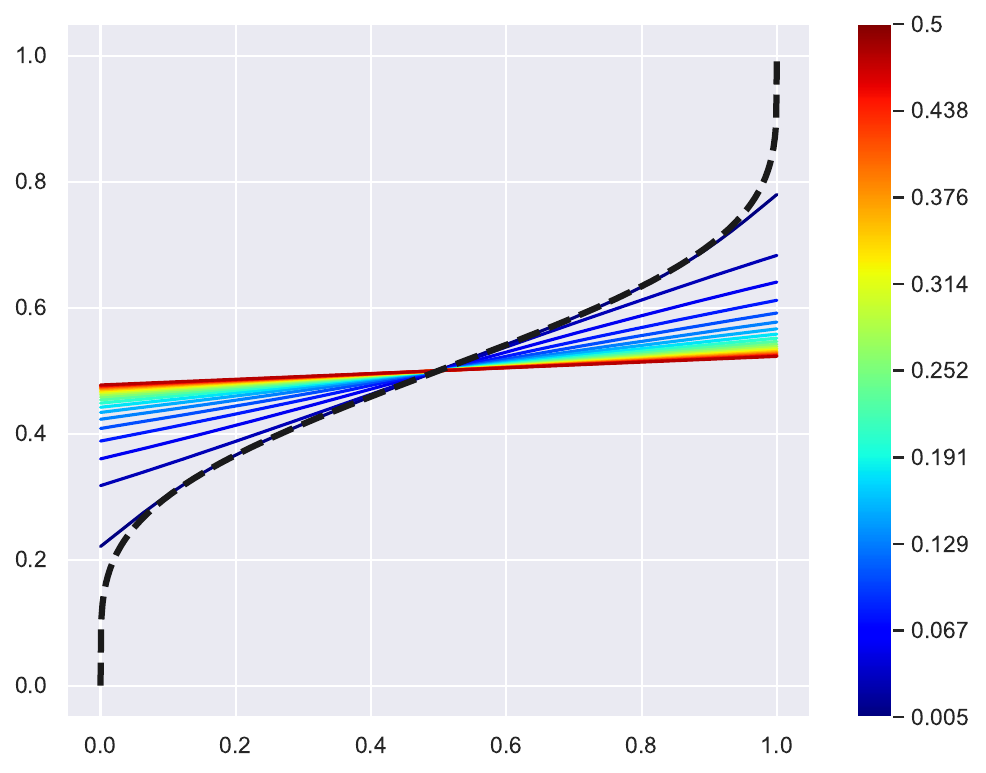}}}
{\subfigure[\mbox{Quadratic cost on the torus}]{\includegraphics[width=0.45 \textwidth,height=0.3\textwidth]{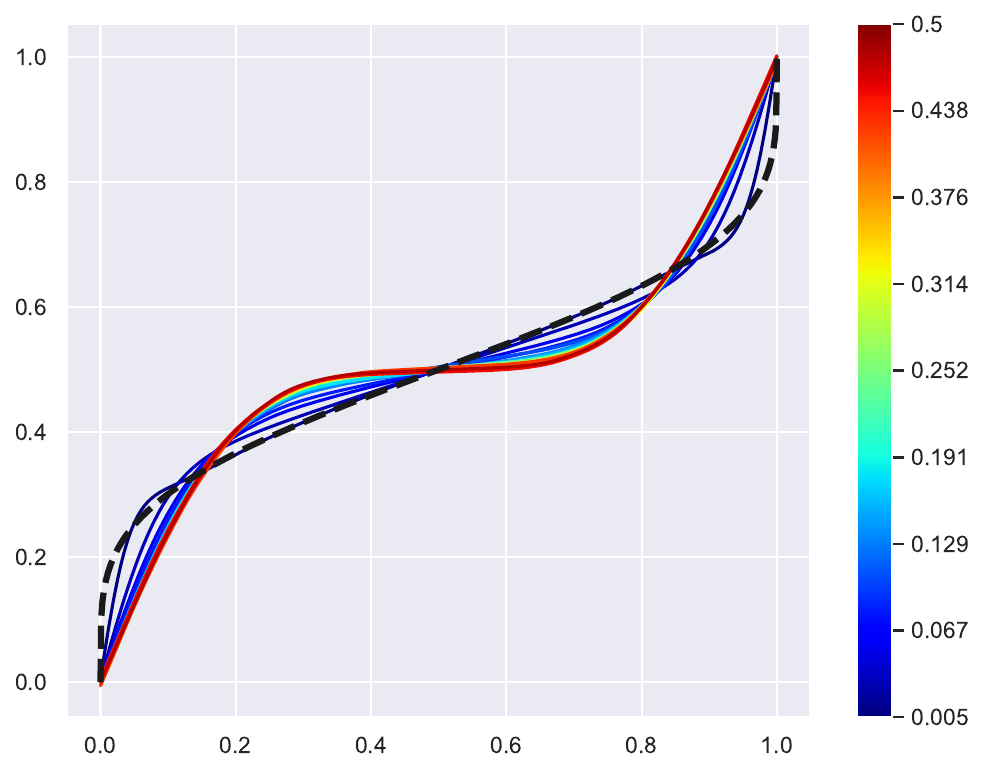}}}

\caption{  Estimators $\widehat{u}_{\varepsilon}^{\, n}$ and $\widehat{Q}_{\varepsilon}^n$ on the first and second lines respectively.
The black and dashed curves are either the un-regularized optimal dual potential $u_0$ or the un-regularized quantile function $Q_0$ of the beta distribution for the standard quadratic cost.
\label{fig:1d_beta_distribution}}
\end{figure}

From now on, let us consider a sample $(Y_1^\ast,\ldots,Y_J^\ast)$ of small size $J=100$ of the same $\mbox{beta}(a,b)$ distribution. 
We illustrate the potential benefits of using regularized OT to obtain a smoother estimator than the usual empirical quantile function $\widehat{Q}_0^J$  defined as the generalized inverse of the empirical cumulative distribution function
\begin{equation*}
\widehat{F}_0^J(x) = \frac{1}{J} \sum_{j=1}^{J} \1_{ \{Y_j^\ast\leq x\}}.
\vspace{-0.2cm}
\end{equation*}
To this end, for various values of $\ee \in [0.005,0.5]$, we compute the two estimators $\widehat{Q}_{\varepsilon}^{n,J}$ from \eqref{eq:hatQ} and $\widetilde{Q}_{\varepsilon}^{m,J}$ from \eqref{eq:hattildeQ} with sequences of $n=m=10^5$ random variables sampled from the discrete measure $\widehat{\nu}_J^\ast$ or the  uniform measure on $[0,1]$ respectively.
In \Cref{fig:1d_beta_discrete_distribution}, we display in logarithmic scale the point-wise mean-squared errors 
\begin{equation*}
\mbox{MSE}(\widehat{Q}_{\varepsilon}^{n,J}(x)) = \EE \left[ |\widehat{Q}_{\varepsilon}^{n,J}(x) -Q_0(x)|^2 \right] \quad \mbox{and} \quad \mbox{MSE}(\widetilde{Q}_{\varepsilon}^{m,J}(x)) = \EE \left[ |\widetilde{Q}_{\varepsilon}^{m,J}(x) - Q_0(x)|^2 \right], 
\end{equation*}
where the above expectations are approximated using Monte-Carlo experiments from 100 repetitions of the above described procedure. The MSE of these regularized estimators is then compared to the MSE of the usual empirical quantile function $\widehat{Q}_0^J$ defined accordingly.
For all values of $\ee$, it can be seen, from  \Cref{fig:1d_beta_discrete_distribution}, that regularization always improves the estimation of $Q_0(x)$ by $ \widehat{Q}_0^J(x)$ around the median location $x=0.5$. For the smallest values of $\ee$, regularization  also improves the estimation of $Q_0(x)$ for $x \in [0.1,0.9]$, and the best results are obtained with the stochastic algorithm based on the FFT.

 \begin{figure}[htbp]
\centering

{\subfigure[ $\log \mbox{MSE}(\widehat{Q}_{\varepsilon}^{n,J}(x)) $]{\includegraphics[width=0.45 \textwidth,height=0.3\textwidth]{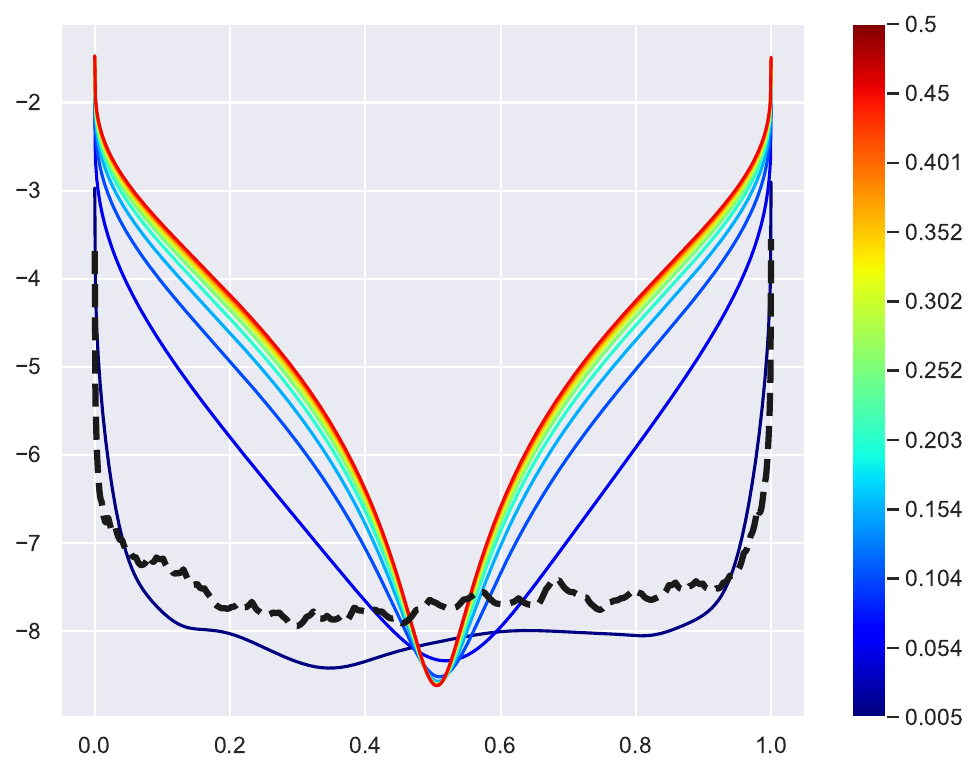}}}
{\subfigure[ $\log \mbox{MSE}(\widetilde{Q}_{\varepsilon}^{m,J}(x)) $]{\includegraphics[width=0.45 \textwidth,height=0.3\textwidth]{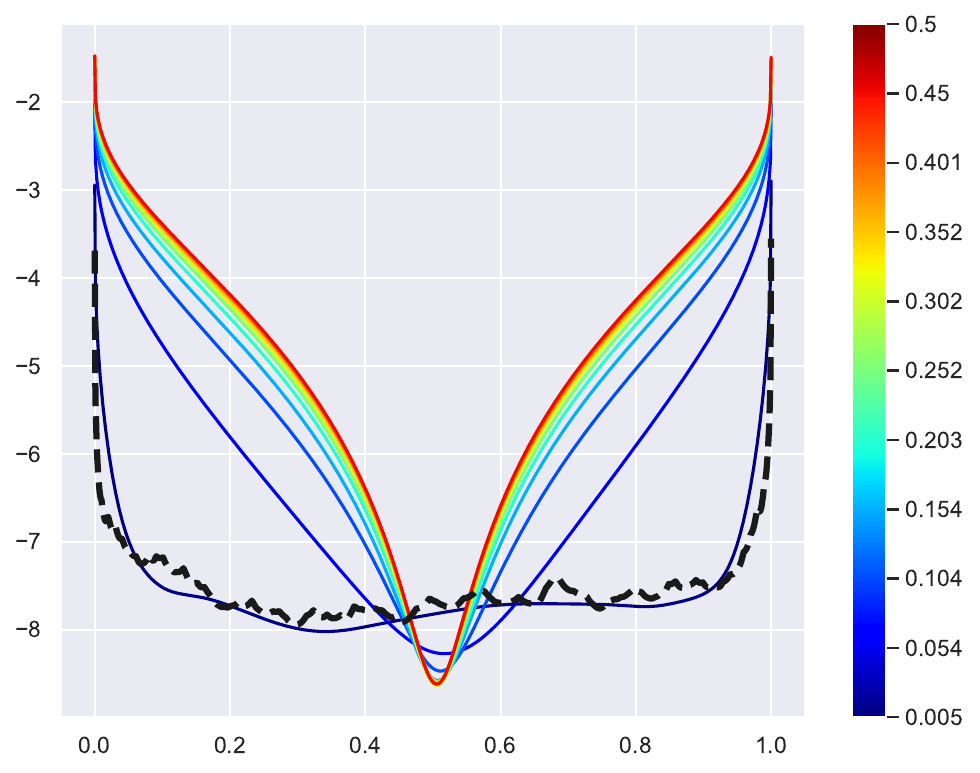}}}


\caption{Point-wise error of the regularized  estimators $\widehat{Q}_{\varepsilon}^{n,J}$ and  $\widetilde{Q}_{\varepsilon}^{m,J}$ for various values of $\ee \in [0.005,0.5]$. The black and dashed curve is the point-wise error of the un-regularized empirical quantile function $\widehat{Q}_0^J$.  \label{fig:1d_beta_discrete_distribution}}
\end{figure}

\subsection{Numerical experiments in dimension $d=2$}

As argued in \cite{Hallin-AOS_2021}, taking as reference the spherical uniform distribution $\mu_{S}$ on the unit ball $\Bb^{d}$ induces different properties for MK quantiles.
Thanks to a change in polar coordinates, one can parametrize on $\Bb^{d}$ instead of $[0,1]^d$.
By definition, a random vector $X$ with spherical uniform distribution is given by $X = R \Phi$ where $R$ and $\Phi$ are independent and drawn uniformly from [0,1] and the unit hypersphere $\Sc^{d-1}$, respectively. In dimension $d=2$, $X$ writes in polar coordinates as
\begin{equation*}
X = \begin{pmatrix}
 R \cos(2\pi \Psi) \\
 R \sin(2\pi \Psi)
 \end{pmatrix} \in \Bb^{2}, 
\end{equation*}
where $(R,\Psi)$ is uniform on $[0,1]^2$. Then, for a function  $u \in L^{1}(\Bb^{d},\mu_{S})$, its parametrization in polar coordinates is given, for all $(r,\psi) \in [0,1] \times [0,1]$, by
 \begin{equation*}
 \overline{u}(r,\psi) = u \begin{pmatrix}
 r \cos(2\pi \psi) \\
 r \sin(2\pi  \psi)
 \end{pmatrix}.
 \end{equation*}
 Hence, by definition of $\mu_S$, the function $\overline{u}$ is an element of $L^{1}\left([0,1]^2, \mu \right)$ where $\mu$ is the uniform distribution on $\XX = [0,1]^2$. Consequently, thanks to this re-parametrization, we propose to solve in the Fourier domain, for $\Lambda = \ZZ^{2} \backslash \{0\}$, the following regularized OT problem 
\begin{equation}
\theta^{\ee} = \argminE_{\theta \in \ell_{1}(\Lambda) }  \overline{H}_\ee(\theta) \hspace{1cm} \mbox{with}  \hspace{1cm}  \overline{H}_\ee(\theta) = \mathbb{E} \left[ \, \overline{h}_\ee(\theta,Y) \right], \label{eq:OTregpolar}
\end{equation}
where $Y =(Y_1,Y_2) \in \R^2$ is a random vector with distribution $\nu$, and $\overline{h}_\ee$ is given by
\begin{align*}
\overline{h}_\ee(\theta,y) = \ee \log \big( \int_{\XX} & \exp \left( \frac{\SL \theta_\lambda \phi_\lambda(r,\psi) -c_y( r, \psi)}{\ee} \right) d\mu(r,\psi) \big) + \ee,
\end{align*}
with $\phi_\lambda(r,\psi)=e^{2\pi i (\lambda_1 r + \lambda_2 \psi)}$ for $\lambda = (\lambda_1,\lambda_2) \in \ZZ^2$, and $c_y$ refers to the quadratic cost,
\begin{equation*}
c_y(r,\psi) = \frac{1}{2 } \left((r \cos(2 \pi \psi) -y_1)^2 +(r \sin(2 \pi \psi) -y_2)^2\right).
\end{equation*}
In order to solve \eqref{eq:OTregpolar}, we adapt the stochastic algorithm \eqref{defthetan} which yields, after $n$ iterations, the sequence $\overline{\theta}_{n}$ and the estimator, in polar coordinates,
\begin{equation}\label{overline_u}
\overline{u}_{\varepsilon}^{\, n}(r,\psi) = \SL \overline{\theta}_{n,\lambda} \phi_\lambda(r,\psi).
\end{equation} 
In practice, we discretize $[0,1]^2$ by choosing equi-spaced radius points $0 \leq r_1 < \ldots < r_{p_1} \leq 1$ and angles $0 \leq \psi_{1} < \ldots \psi_{p_2} < 1$ which results in taking a grid of $p=p_1  p_2$ points
 \begin{equation*}
 \XX_{p} = \left\{ (r_{\ell_1},\psi_{\ell_2})_{(\ell_1,\ell_2) \in \{1,p_1 \} \times \{1, p_2 \}} \right\} \subset [0,1]^2.
 \end{equation*}
Finally, the stochastic algorithm \eqref{defthetan} is implemented on this polar grid using the weight sequence $w_{\lambda} = 1$ for all $\lambda \in \Lambda_{p}$ that is with $\alpha = 0$. Of course, \Cref{hyp:Winv} is always verified if $(w_\lambda) \equiv (1,1, \cdots, 1, 0,0,\cdots)$. This is motivated by the fact that choosing $w_\lambda = \| \lambda \|^{-\alpha}$ with $\alpha \geq 1$ would impose periodic constraints on the dual potentials $\bar{u}(r,\psi)$ along the radius coordinate.
However, as shown by the following numerical experiments, an optimal dual potential typically does not satisfy the polar periodic conditions $\bar{u}(0,\psi) = \bar{u}(1,\psi)$ for all $\psi \in [0,1]$.
The counterpart of $\widehat{Q}_\ee^n$ in \eqref{eq:hatQ} directly follows from \eqref{overline_u}, that is 
\begin{equation}
\overline{Q}^n_{\, \varepsilon}(x)   = \sum_{j=1}^{n} \overline{F}_{j}(x) Y_{j}
 \hspace{1cm} \mbox{where}  \hspace{1cm}  
 \overline{F}_{j}(x) = \frac{ \exp \Bigl( \dfrac{ (\overline{u}_{ \varepsilon}^{\,n})^{c,\varepsilon}(Y_j) - c(x,Y_j) }{\varepsilon} \Bigr) }{\sum_{\ell = 1}^{n} \exp \Bigl( \dfrac{ (\overline{u}_{\varepsilon}^{\,n})^{c,\varepsilon}(Y_\ell)  - c(x,Y_\ell) }{\varepsilon} \Bigr)},  \label{eq:lineQ}
\end{equation} 
where the integral in the computation $(\overline{u}_{ \varepsilon}^{\,n})^{c,\varepsilon}(\cdot)$ is approximated with the polar grid $\XX_p$.  In what follows, we report numerical experiments for the banana-shaped distribution $\nu$ considered in  \cite{chernozhukov2015mongekantorovich}. It corresponds to sampling $Y$ as the random vector
 \begin{equation*}
 Y =  \begin{pmatrix}
 U + R \cos(2\pi \Phi) \\
 U^2 + R \sin(2\pi  \Phi)
 \end{pmatrix},
 \end{equation*}
where $U$ is uniform on $[-1,1]$, $\Phi$ is uniform on $[0,1]$, $R = 0.2 Z (1-(1-|U|)/2$ with $Z$ uniform on $[0,1]$, and $U,\Phi$ and $Z$ independent. In these simulations, the random variable $Y$ is also centered and scaled so that it takes its values within the subset $[-0.6,0.6] \times [-0.4,0.5] \subset [0,1]^2$.

We first consider a sample $Y_1^\ast,\ldots,Y_J^\ast$ of size $J=10^3$ that is held  fixed and displayed in \Cref{fig:2d:comparisonp1p2}. Then, we draw $n = 10^5$ random variables $Y_1,\ldots,Y_n$ from the associated discrete distribution $\widehat{\nu}_{J}^*$, and we run the stochastic algorithm  \eqref{defthetan} with different sizes $(p_1,p_2) = (10,100)$ and  $(p_1,p_2) = (100,1000)$ for the discretization $\XX_{p}$.
Note that the cost of each iteration of the stochastic algorithm is of order $\bigO\left( p \log(p)\right)$ for $p=p_1p_2$. Therefore, the choice of discretization of the polar coordinates greatly influences the computational cost of the algorithm.
In \Cref{fig:2d:comparisonp1p2}, we display the resulting regularized dual potentials $\widehat{u}_{\varepsilon}^{\, n}$ in cartesian coordinates for $\ee = 0.005$. 
We also draw the resulting MK contour quantiles of level $r = 0.5$ for each choice of discretization. It can be seen that the resulting  MK contour quantiles are very similar with a much lowest computational cost for the discretization of size  $(p_1,p_2) = (10,100)$. 
\begin{figure}[htbp]
\centering
{\subfigure[Quantile contour for $r=0.5$ ; $(p_1,p_2) = (10,100)$]{\includegraphics[width=0.3 \textwidth,height=0.3\textwidth]{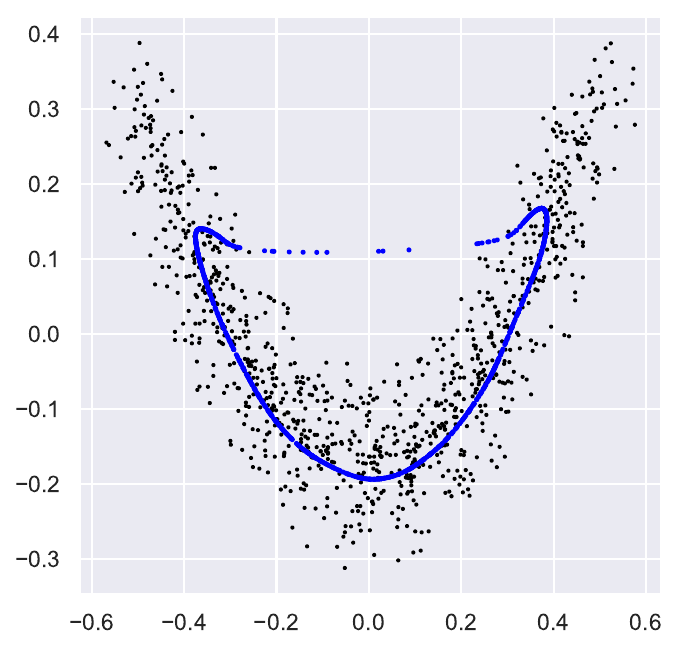}}}
{\subfigure[$\widehat{u}_{\varepsilon}^{\, n}(r \cos(2 \pi \psi), r \sin(2 \pi \psi))$ ; $(p_1,p_2) = (10,100)$]{\includegraphics[width=0.3 \textwidth,height=0.3\textwidth]{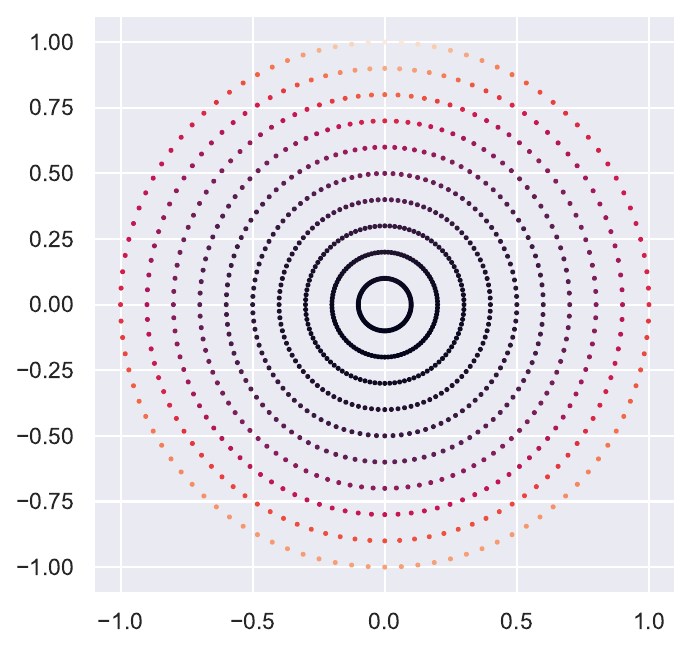}}}
{\subfigure[ $\overline{u}_{\varepsilon}^{\, n}(r,\psi)$  - $(p_1,p_2) = (10,100)$]{\includegraphics[width=0.35 \textwidth,height=0.12\textwidth]{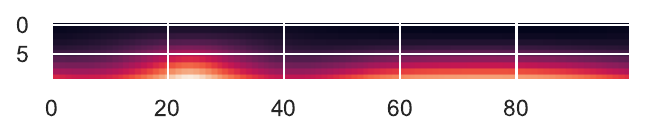}}}

{\subfigure[Quantile contour for $r=0.5$ ; $(p_1,p_2) = (100,1000)$]{\includegraphics[width=0.3 \textwidth,height=0.3\textwidth]{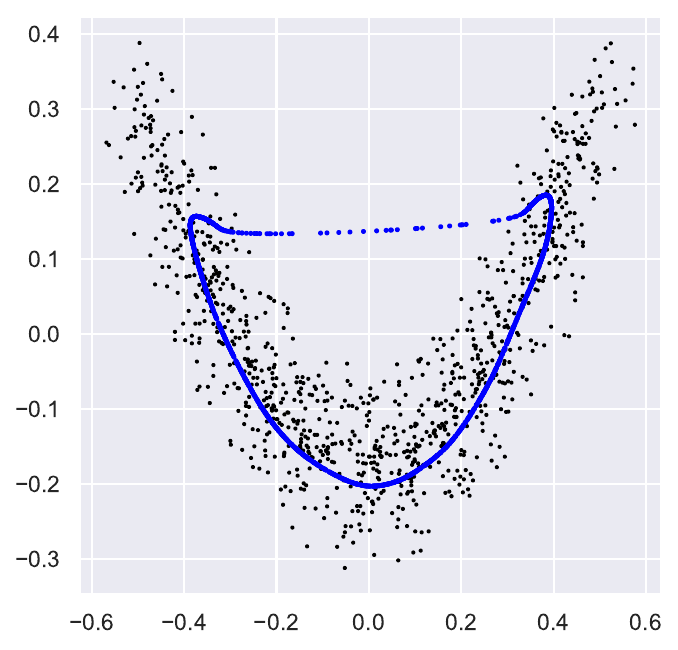}}}
{\subfigure[$\widehat{u}_{\varepsilon}^{\, n}(r \cos(2 \pi \psi), r \sin(2 \pi \psi))$ ; $(p_1,p_2) = (100,1000)$]{\includegraphics[width=0.3 \textwidth,height=0.3\textwidth]{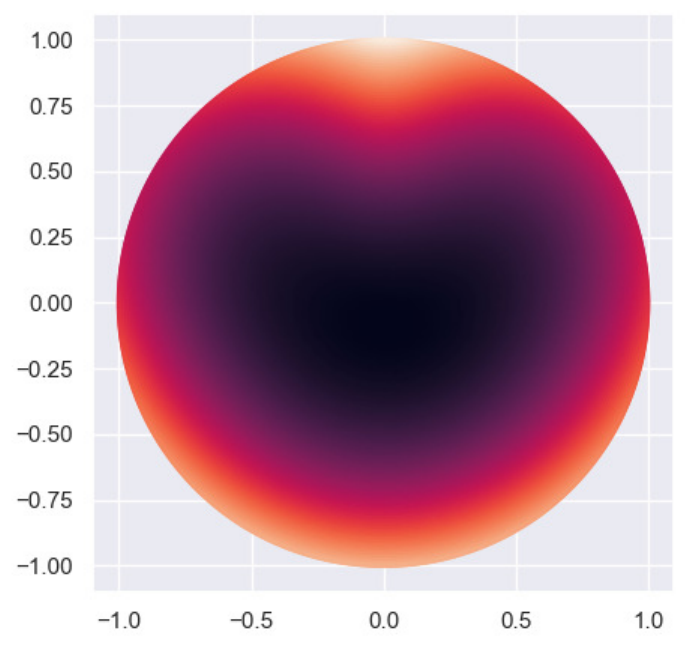}}}
{\subfigure[ $\overline{u}_{\varepsilon}^{\, n}(r,\psi)$ - $(p_1,p_2) = (100,1000)$]{\includegraphics[width=0.35 \textwidth,height=0.12\textwidth]{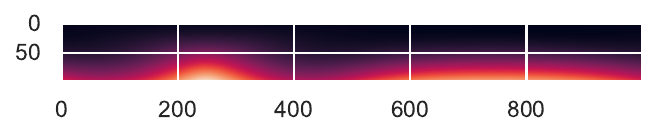}}}

\caption{The blue curves are  regularized MK quantile contours  at level $r=0.5$ for $\ee = 0.005$ from the discrete measure $\hat{\nu}_{J}$ (displayed with black points) using  two different  discretizations $(p_1,p_2) = (10,100)$ (first row) and $(p_1,p_2) = (100,1000)$ (second row). The second (resp. thrid) columns represent the values of the regularized  dual potentials   in cartesian (resp. polar) coordinates for each choice of discretization.\label{fig:2d:comparisonp1p2}}
\end{figure}

\Cref{fig:2d:comparaison_algos} contains a comparison of the convergence between our FFT-based scheme \eqref{eq:hatQ} and \eqref{eq:hattildeQ}, based on the stochastic gradient descent from \cite{bercu2020asymptotic}, that we refer to as the regularized SGD. The reference distribution is taken to be the spherical uniform. 
Also, we compare these regularized approaches (using $\ee = 0.005$) with classical un-regularized ones. 
To this end, we implement a subgradient descent for un-regularized OT, namely the same Robbins-Monro scheme as \eqref{Semi-dualdisc0} with $\ee=0$, that is a semi-discrete scheme advocated in \cite{chernozhukov2015mongekantorovich}[Section 4]. Finally, we use the OT network simplex solver from the Python library \cite{POT} to compute the solution of un-regularized OT between two empirical discrete distributions with supports  $\XX_p = \{x_1,\ldots,x_p\}$  and $(Y_1^\ast,\ldots,Y_J^\ast)$. We first consider a sample $Y_1^\ast,\ldots,Y_J^\ast$ of size $J=10^4$ that is held fixed. For our FFT approach, we let $p_1 = 20$, $p_2=500$, so that $p = p_1p_2 = 10^4$. 
For the three iterative schemes, the number of iterations varies between $10^4$, $10^5$ and $10^6$. 
This corresponds, for our FFT approach, to a stochastic algorithm with $1,10$ and $100$ epochs, whereas the other approaches sample from the reference distribution $\mu_S$. 
The first line of \Cref{fig:2d:comparaison_algos} contains the corresponding quantile contours of order $r=0.5$ for each of these methods, for several number of iterations. The colored dots are obtained by transporting points of radius $r=0.5$, while the lines between them are visual artefacts. 
Unlike regularized estimators, the quantile function estimated from an un-regularized semi-discrete scheme is restricted to take its values in the set of observations $( Y_1^\ast,\ldots,Y_J^\ast )$.
On another hand, with the simplex solver, the obtained empirical quantile map is not a function, rather a collection of points.
The use of stochastic algorithms is more targeted to this task. Still, it is represented here as a benchmark, indicating where the quantile contours shall be. 
Furthermore, the second line of \Cref{fig:2d:comparaison_algos} deals with convergence depending on the number of iterations. 
As customary, we consider a recursive estimation of the values of our objectives, respectively $H_\ee$ for \eqref{Se}, $\widetilde{H}_\ee$ for \eqref{Semi-dualdisc0} and $\widetilde{H}_0$ for \eqref{Semi-dualdisc0} with $\ee=0$. These objectives are recursively estimated along the iterations by gradual averaging in order to account for convergence,  as proposed in \cite{bercu2020asymptotic}. For $J = p=10^4$, the computational cost at each iteration of the two regularized procedures is of the same order. 
It can be seen that the un-regularized SGD has not converged with $10^6$ iterations, whereas the regularized approaches 
\eqref{eq:hatQ} and \eqref{eq:hattildeQ}
have similar convergence behavior. 
Together with the first line of \Cref{fig:2d:comparaison_algos}, these results illustrate that entropically regularized methods converge faster towards a more suitable solution.

\begin{figure}[htbp]
\centering
{\subfigure[Quantile contour at level $r=0.5$ with $n=10^4$ iterations]{\includegraphics[width=0.3 \textwidth,height=0.3\textwidth]{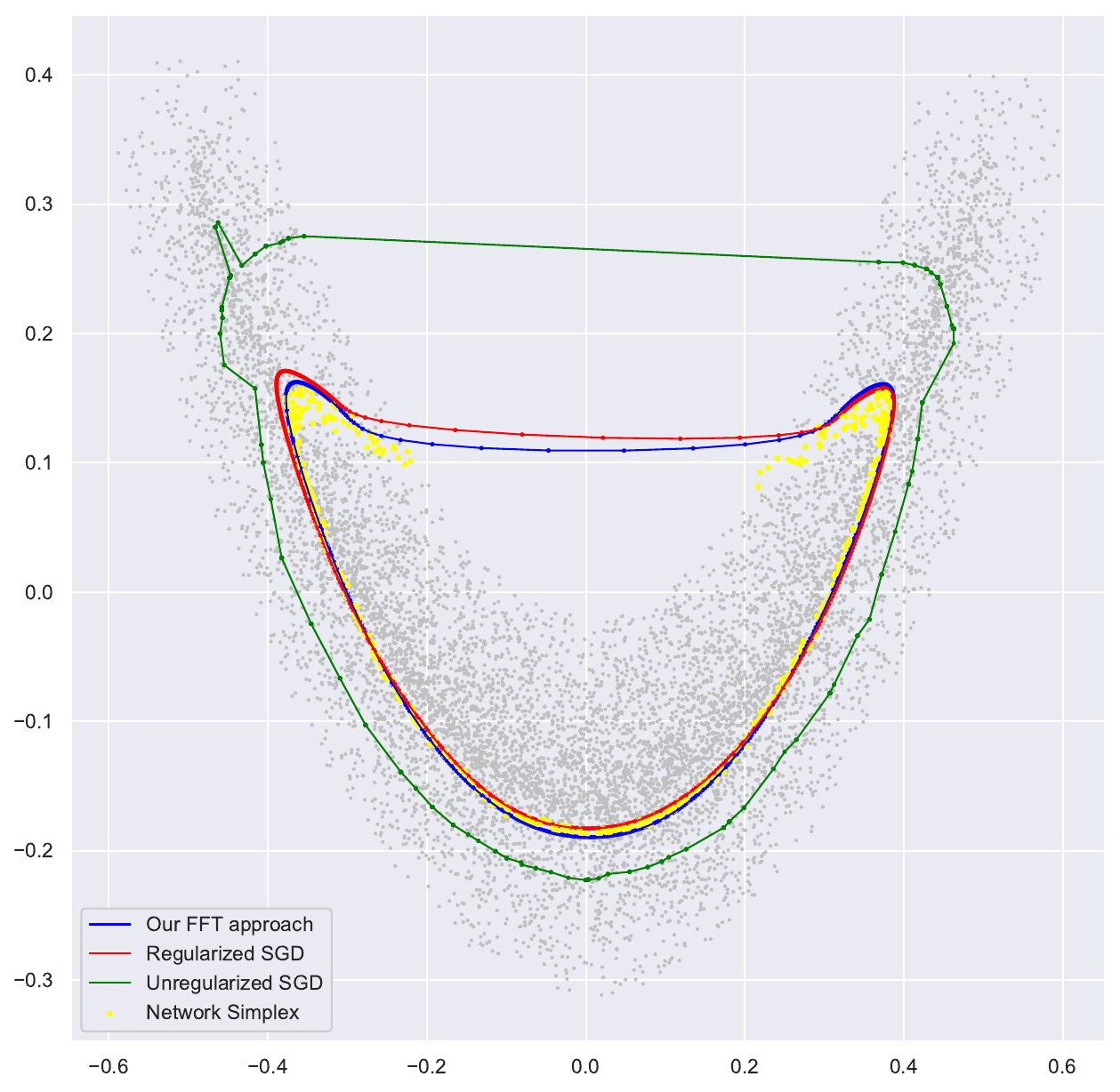}}}
{\subfigure[Quantile contour at level $r=0.5$ with $n=10^5$ iterations]{\includegraphics[width=0.3 \textwidth,height=0.3\textwidth]{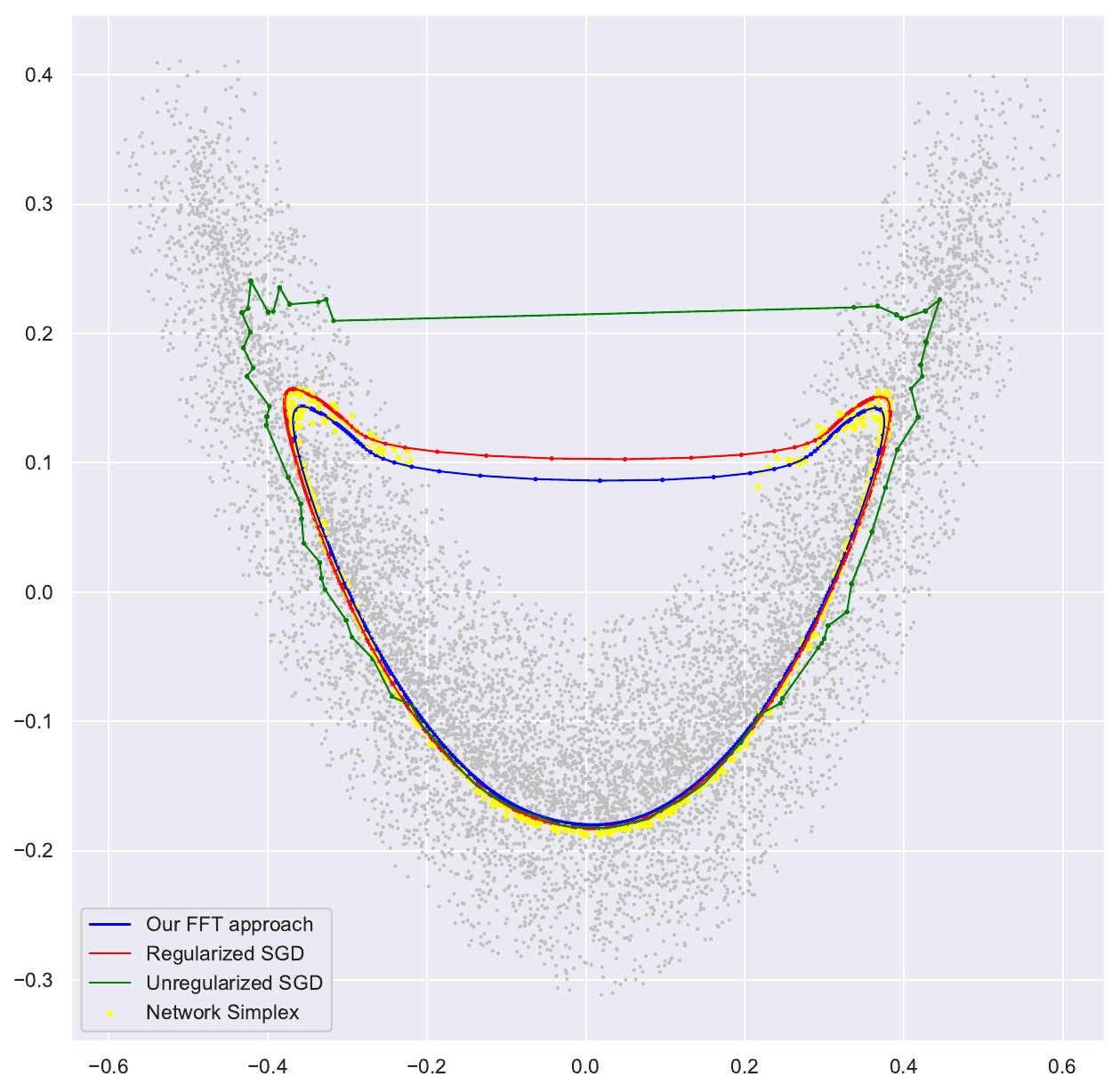}}}
{\subfigure[Quantile contour at level $r=0.5$ with $n=10^6$ iterations]{\includegraphics[width=0.3 \textwidth,height=0.3\textwidth]{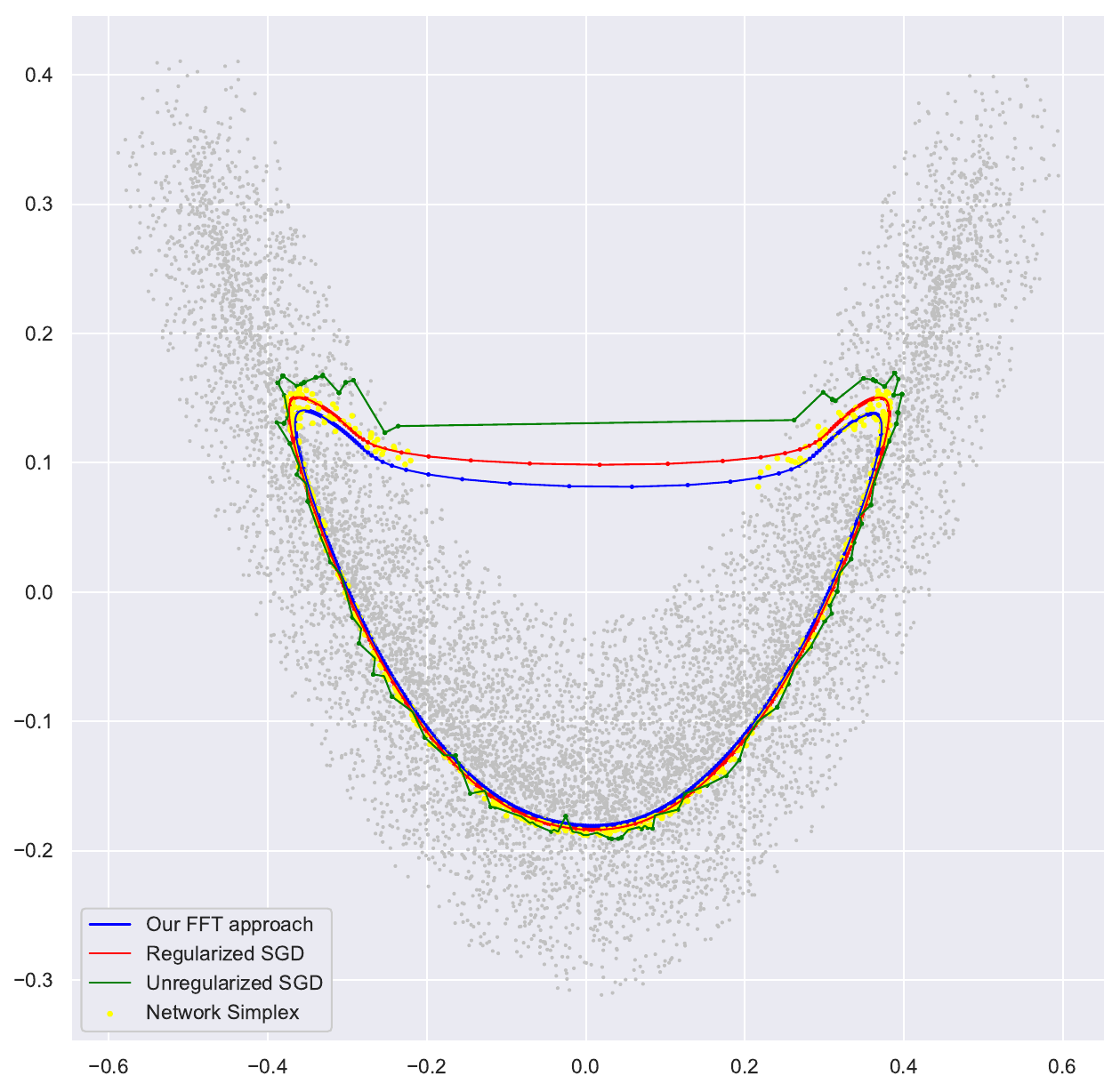}}}

{\subfigure[ Convergence of our FFT-based scheme]{\includegraphics[width=0.3 \textwidth,height=0.18\textwidth]{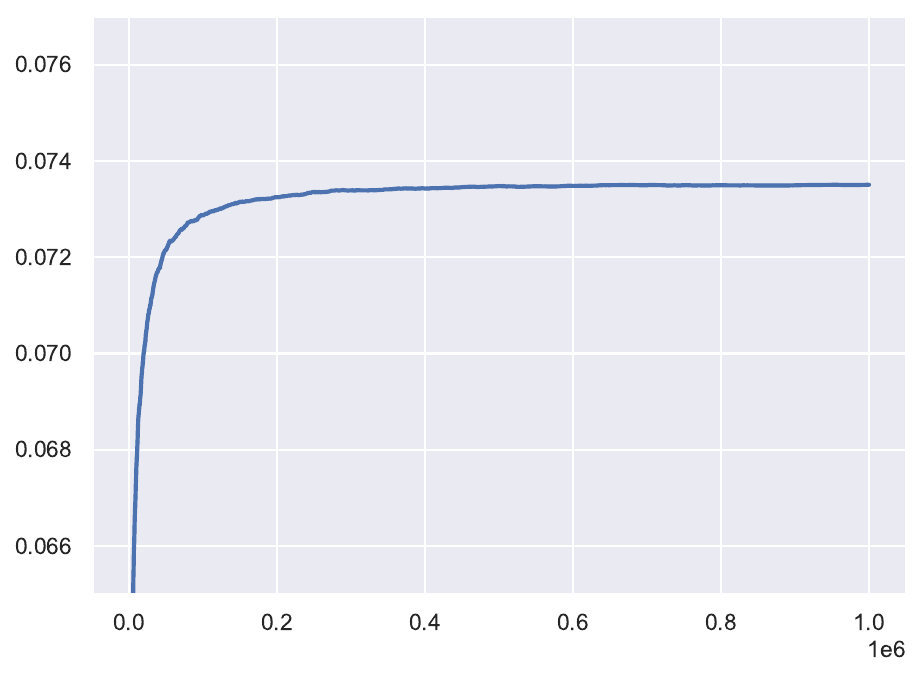}}}
{\subfigure[ Convergence of the regularized SGD]{\includegraphics[width=0.3 \textwidth,height=0.18\textwidth]{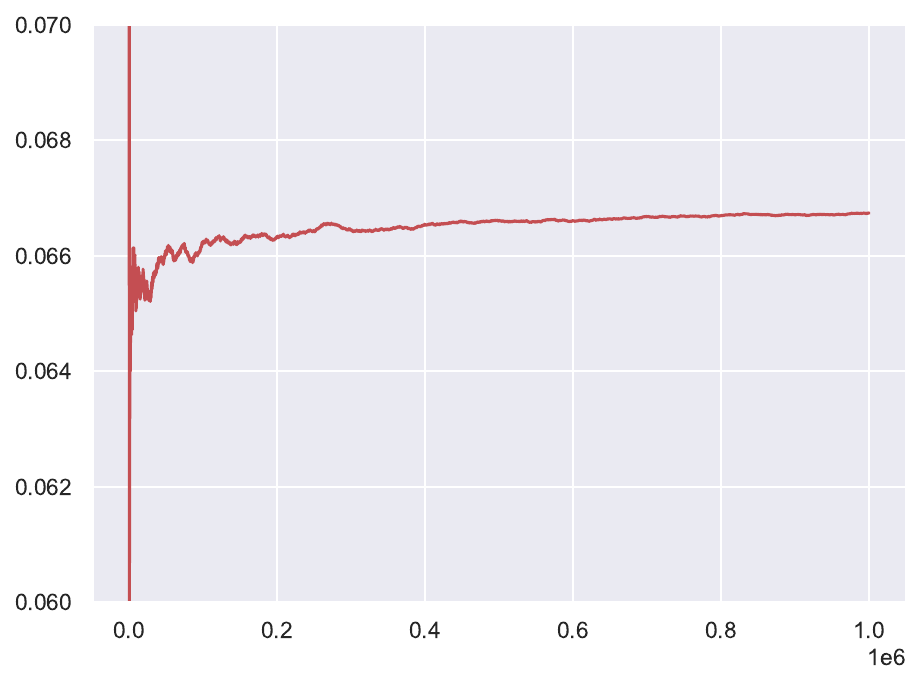}}}
{\subfigure[ Convergence of the unregularized SGD ]{\includegraphics[width=0.3 \textwidth,height=0.18\textwidth]{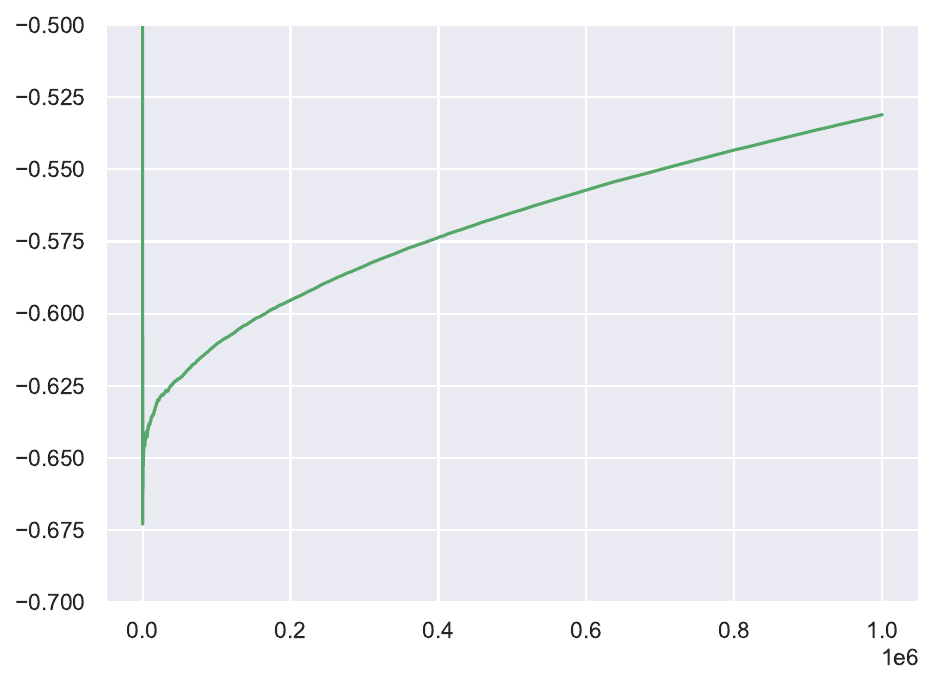}}}

\caption{ Comparison between regularized (with $\ee = 0.005$) and unregularized approaches. }\label{fig:2d:comparaison_algos}

\end{figure}


We finally propose a last numerical experiment to highlight the behavior of EOT when varying the regularization parameter $\epsilon$. We chose to draw $n=10^7$ random variables $Y_1,\ldots,Y_n$ from the banana-shaped distribution, and we ran the stochastic algorithm \eqref{defthetan} for the discretization $(p_1,p_2) = (10,1000)$. Doing so, the obtained sample is very close to the true density, and the various resulting contours only depend on $\ee$. In \Cref{fig:2d:contourMKvariouseps}, we display the resulting regularized MK quantile contours of levels $r \in \{0.2,0.3,\ldots,1\}$ for different values of $\ee \in [0.002,0.5]$. This visualization warns on the choice of the regularization parameter that must be chosen small enough, as usual with EOT. Note that, for $n=10^7$ observations, we have not been able to implement the Sinkhorn algorithm. Moreover, the cost at each iteration   of either regularized or un-regularized SGD being $\bigO(n)$, these algorithms are much slower to converge than our approach.

\begin{figure}[htbp]
\centering
{\subfigure[$\ee = 0.5$]{\includegraphics[width=0.32 \textwidth,height=0.32\textwidth]{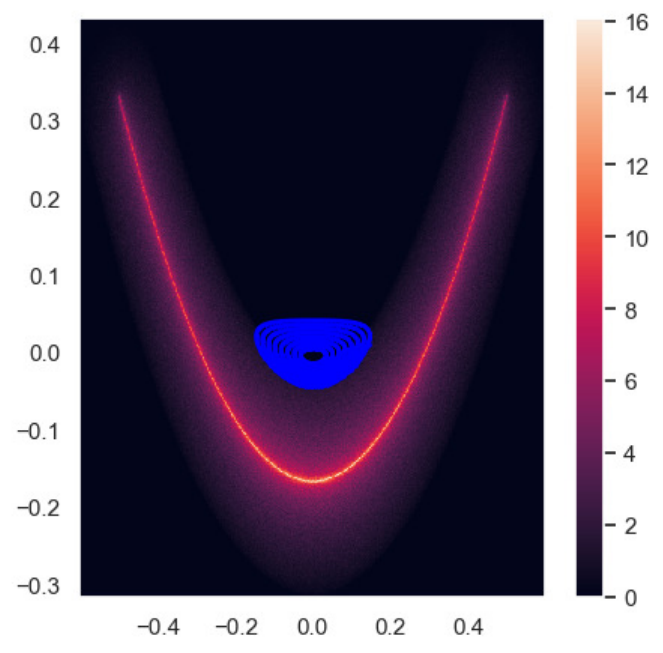}}}
{\subfigure[$\ee = 0.1$]{\includegraphics[width=0.32 \textwidth,height=0.32\textwidth]{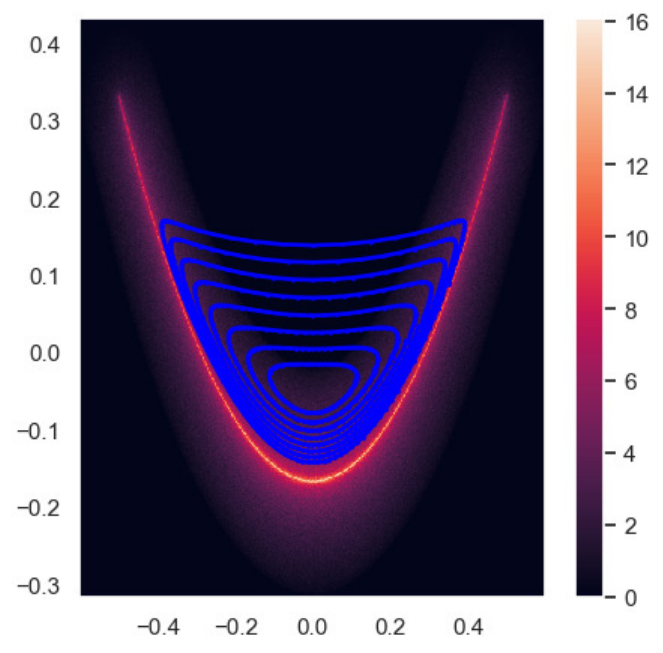}}}
{\subfigure[$\ee = 0.05$]{\includegraphics[width=0.32 \textwidth,height=0.32\textwidth]{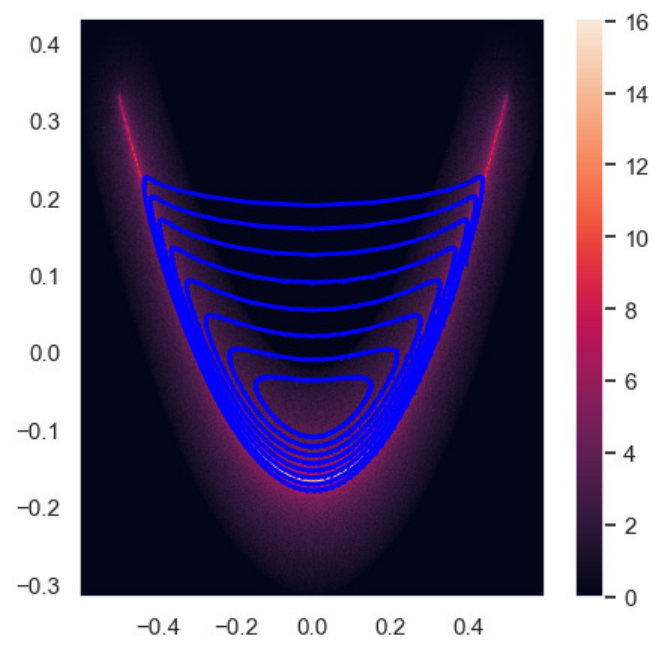}}}

{\subfigure[$\ee = 0.01$]{\includegraphics[width=0.32 \textwidth,height=0.32\textwidth]{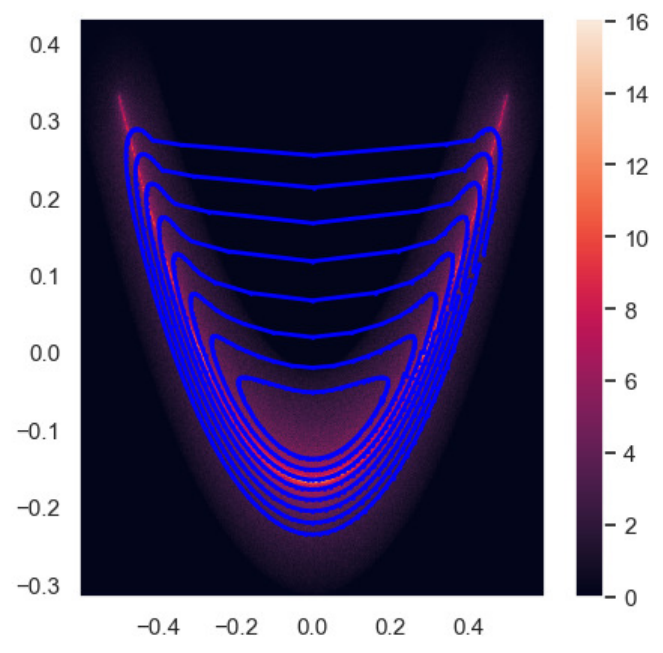}}}
{\subfigure[$\ee = 0.005$]{\includegraphics[width=0.32 \textwidth,height=0.32\textwidth]{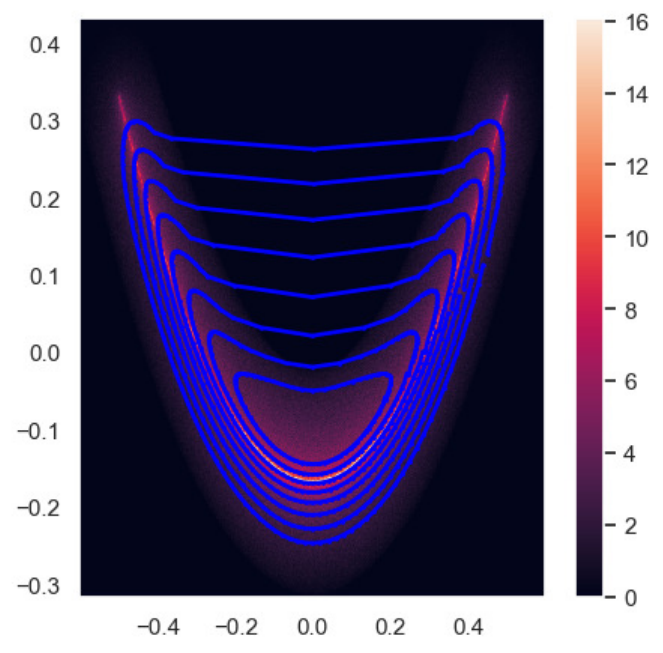}}}
{\subfigure[$\ee = 0.002$]{\includegraphics[width=0.32 \textwidth,height=0.32\textwidth]{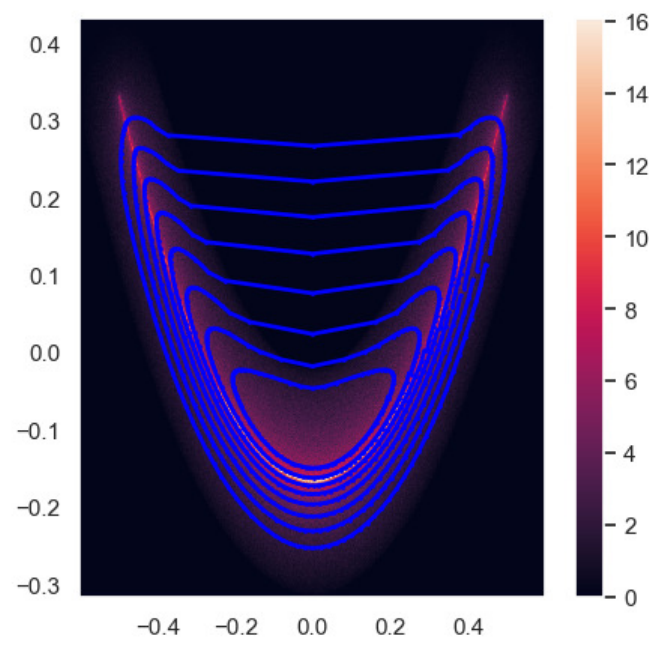}}}

\caption{In all the figures, the image at the background represents a density histogram from the empirical measure $\hat{\nu}_{n} = \frac{1}{n} \sum_{j=1}^{n}  \delta_{Y_j}$ where  $Y_1,\ldots,Y_n$ are sampled from the banana-shaped distribution  with $n=10^7$. The blue curves correspond to  regularized MK quantile contours of levels $r \in \{0.2,0.3,\ldots,1\}$ for $\ee \in [0.002,0.5]$. \label{fig:2d:contourMKvariouseps}}
\end{figure}

\section{Conclusion and perspectives}\label{CCL}

Throughout the paper, we advocated the use of the entropic map for MK quantiles' estimation. Indeed, it is a smooth approximation of an OT map and benefits from the crucial cyclical monotonicity together with computational benefits of EOT.
Our new stochastic algorithm for the continuous OT problem showed potential improvement in terms of numerical complexity, because it is independent, at each iteration, from the size of the observed sample. 
Nonetheless, our implementation of the FFT may become intractable in high dimensions.
Because of the known decay of Fourier coefficients,
one can hope that more sophisticated FFTs could alleviate this, see $e.g.$ \cite{potts2021approximation}, but this is beyond the scope of the present paper. 

Our convergence study based on random iterative schemes extends results from \cite{bercu2020asymptotic} to the continuous setting instead of the semi-discrete setting. Minimax convergence rates of un-regularized estimators of OT maps have been obtained in  recent works  \cite{ghosal2021multivariate,hutter_rigollet}.  Hence, it would be interesting to extend our analysis to the study of the rate of convergence of our regularized estimator. This is an interesting  challenge that is left for future work.

As argued $e.g.$ in \cite{Hallin-AOS_2021}, our assumption of finite second-order moment for $\nu$ may be too restrictive for multivariate quantiles. 
In the seminal paper \cite{Hallin-AOS_2021}, using McCann's theorem \cite{McCannThm}, the definition of Monge-Kantorovich quantiles have been extended as a push-forward map between the reference and the target measures, that is also the gradient of a convex function. 
In order to get rid of this moment assumption using EOT, future work may consider 
the insightful results from \cite{ghosal2022stability},
as their notion of \textit{cyclically invariant} coupling can always yield a mapping by barycentric projection, which coincides with $Q_\ee$ if the cost $c$ belongs to  $L^1(\mu\otimes\nu)$.


\

\textbf{Funding:} The authors gratefully acknowledge financial support from the Agence Nationale de la Recherche  (MaSDOL grant ANR-19-CE23-0017). J\'er\'emie Bigot is a member of Institut Universitaire de France (IUF), and this work has also been carried out with financial support from the IUF.

\newpage
$$
\textbf{\Large{Appendix}}
$$

\appendix
\section{Proofs of the main results}
\label{Appendix:proof}

The proofs of Proposition \ref{prop:grad}, Proposition \ref{prop:Hess} and Proposition  \ref{prop:convexity} are given in supplementary materials, see \Cref{sup:prop:grad}, \Cref{sup:prop:Hess} and \Cref{sup:prop:convexity}.  We shall now proceed to the proofs of the main results of the paper.

%
%

\subsection{Proof of Proposition \ref{prop:reg_Hee}} 
For  $\tht \in \overline{\ell_1}(\Lambda)$ and $t\in [0,1]$, we denote $\tht_t = \tht^\ee + t(\tht- \tht^\ee)$ and we define the function $\varphi(t) = H_\ee(\tht_t).$ Then, we deduce from a second order Taylor expansion of $\varphi$ with integral remainder that
\begin{equation}
\varphi(1) = \varphi(0) + \varphi'(0) + \int_{0}^1(1-t)\varphi''(t)dt.
\label{Taylorvarphi}
\end{equation}
However, we clearly have
\begin{equation}
\varphi'(t)= DH_\ee(\tht_t)[ \tht - \theta^\ee]\hspace{1cm} \text{ and } \hspace{1cm}  \varphi''(t) =    D^2 H_\ee(\tht_t)[\tht - \theta^\ee, \tht - \theta^\ee].
\end{equation}
Consequently, as $\varphi'(0) = DH_\ee(\tht^\ee)[ \tht - \theta^\ee]=0$, \eqref{Taylorvarphi} can be rewritten as
\begin{equation}
\label{TaylorvarphiReplaced}
H_\ee(\tht) - H_\ee(\theta^\ee) = \int_{0}^1 (1-t)  D^2 H_\ee(\tht_t)[\tht - \theta^\ee, \tht - \theta^\ee] dt.
\end{equation}
Therefore, \eqref{eq:boundHeps} immediately follows from \eqref{eq:opnormHessH} and \eqref{TaylorvarphiReplaced}.
It only remains to prove \eqref{eq:Grad_vs_hess}.
Our strategy is to adapt to the setting of this paper the notion of self-concordance as introduced in \cite{Bach_2010,Bach14} and used in \cite{bercu2020asymptotic,BB_GN} to study the statistical properties of stochastic optimal transport.

\begin{lem}\label{lem:self_conco}
For  $\theta \in \overline{\ell_1}(\Lambda)$ and for all $0<t<1$, denote $\tht_t = \tht^\ee + t(\tht - \tht^\ee)$. Then, the function  $\varphi(t)=H_\ee(\tht_t)$ verifies the self-concordance property 
\begin{equation}\label{self_conco}
\vert\varphi'''(t)\vert \le \frac{2}{\ee}\Vert \tht - \tht^\ee \Vert_{\ell_1} \varphi''(t).
\end{equation}
\end{lem}

\begin{proof}
For a fixed $y\in\YY$, let $\phi(t) = h_\ee(\tht_t,y)$. Firstly, we show that $\phi(t)$ verifies the self-concordance property.  From the chain rule,
we obtain that
\begin{align*}
\phi'(t) &= Dh_\ee(\tht_t,y)[ \tht - \theta^\ee],\\
\phi''(t) &=   D^2 h_\ee(\tht_t,y)[\tht - \theta^\ee, \tht - \theta^\ee],\\
\phi'''(t) &=    D^3 h_\ee(\tht_t,y)[\tht - \theta^\ee, \tht - \theta^\ee,\tht - \theta^\ee],
\end{align*}
where $D^3 h_\ee$ denotes the third order Fréchet derivative of $h_\ee(\cdot,y)$. 
It follows from
\eqref{Gradh} that
\begin{equation}\label{dphit}
\phi'(t) =  \int_\XX S(x) F_{\tht_t, y}(x)  d\mu(x) \hspace{1cm} \mbox{where} \hspace{1cm} S(x) = \SL(\tht_\lambda-\tht_\lambda^\ee)\phi_\lambda(x).
\end{equation}
Similarly, \eqref{eq:D2h_ee} yields
\begin{equation}\label{phi2nde}
\ee \phi''(t) = \int_\XX S(x)^2 F_{\tht_t,y}(x) d\mu(x) - \left( \int_\XX S(x) F_{\tht_t,y}(x) d\mu(x) \right)^2.
\end{equation}
Hereafter, denoting by $Z_t$ the random variable with density $F_{\tht_t, y}$ with respect to  $\mu$, it appears that $\ee \phi''(t) = \EE [ S(Z_t)^2 ]-\EE [ S(Z_t) ]^2 = \EE [(S(Z_t) - \EE[S(Z_t)])^2]$, that is 
\begin{equation}\label{phi2nde2}
\ee \phi''(t) =   \int_\XX \left(  S(x) - \int_\XX S(z) F_{\tht_t,y}(z) d\mu(z) \right)^2 F_{\tht_t,y}(x) d\mu(x).
\end{equation} 
%
Furthermore, using \eqref{dtht_Ftheta} in the derivation of \eqref{phi2nde}, we have that
\begin{equation*}
\ee\phi'''(t)=\int_\XX S(x)^2 \frac{d}{dt}F_{\tht_t,y}(x)  d\mu(x) - 2\left( \int_\XX S(x) \frac{d}{dt}F_{\tht_t,y}(x) d\mu(x) \right)\int_\XX S(x) F_{\tht_t,y}(x) d\mu(x),
\end{equation*}
which yields
\begin{align*}
\ee^2\phi'''(t) &= \int S^3(x) F_{\tht_t,y}(x) d\mu(x) - \left(\int S^2(x) F_{\tht_t,y}(x) d\mu(x) \right)\left( \int S(x) F_{\tht_t,y} (x)d\mu(x)\right) \\
& - 2  \int S(x) F_{\tht_t,y} (x)d\mu(x) \left[  \int S^2(x) F_{\tht_t,y}(x)d\mu(x) - \left( \int S(x) F_{\tht_t,y}(x) d\mu(x) \right)^2 \right].
\end{align*}
Consequently,
\begin{equation*}
\ee^2\phi'''(t) = m_3 - m_2m_1 - 2m_1(m_2-m_1^2) = m_3 - 3m_2m_1 + 2m_1^3,
\end{equation*}
where $m_i$ stands for the $i$-th moment of the distribution of $S(Z_t)$. 
Then, one recognizes the formula for the cumulant of order 3 of a random variable, and so the above equality can be factorized as 
\begin{align}\label{phi3ieme}
\ee^2 \phi'''(t) = \EE[\left(S(Z_t) - m_1\right)^3] = \int (S(x)-m_1)^3 F_{\tht_t,y}(x) d\mu(x).
\end{align}
Thanks to the connection between $\ee\phi''(t)$ and the  variance term in \eqref{phi2nde2}, \eqref{phi3ieme} leads to 
\begin{equation*}
\ee \vert \phi'''(t) \vert \le \supl_{x \in \XX} \vert S(x)-m_1 \vert \phi''(t).
\end{equation*}
It is easy to see that $\vert S(x)-m_1 \vert \le \vert S(x)\vert + \vert m_1 \vert \le 2\Vert \tht-\tht^\ee\Vert_{\ell_1}$. Hence 
\begin{equation}\label{almostSelfConcord}
\vert\phi'''(t)\vert \le \frac{2}{\ee}\Vert \tht-\tht^\ee\Vert_{\ell_1} \phi''(t).
\end{equation}
Finally, given that $\varphi(t) = H_\ee(\tht_t) = \int_\YY h_\ee(\tht_t,y) d\nu(y) = \int_\YY  \phi(t) d\nu(y) $,
\eqref{almostSelfConcord} induces the self-concordance property of $\varphi$.
\end{proof}

We are now in a position to prove inequality \eqref{eq:Grad_vs_hess}. Denote $\delta = 2\Vert \tht - \tht^\ee \Vert_{\ell_1}/\ee$. It follows from inequality \eqref{self_conco} that, for all $0 < t < 1$, 
$
\vert\varphi'''(t)\vert \le \delta \varphi''(t),
$
which leads to
$
\frac{\varphi'''(t)}{\varphi''(t)} \ge - \delta.
$
%
By integrating the above inequality  between $0$ and $t$, we obtain that $
\log \varphi''(t) - \log \varphi ''(0) \geq - \delta t, $
which means that
$
\frac{\varphi''(t)}{\varphi''(0)} \ge e^{- \delta t}.
$
Integrating once again the previous inequality between $0$ and $1$, we obtain that
\begin{equation}
\label{LastProofPropSelfconco}
\varphi'(1) - \varphi'(0) \ge \left(\frac{1 - e^{- \delta}}{\delta}\right)\varphi''(0).
\end{equation}
Finally, as $\varphi'(1) = DH_\ee(\tht)(\tht-\tht^\ee)$, $\varphi'(0) = 0$ and $\varphi''(0) = D^2H_\ee(\tht^\ee)[\tht-\tht^\ee,\tht-\tht^\ee]$,
inequality \eqref{eq:Grad_vs_hess} holds, which completes the proof of \Cref{prop:reg_Hee}.

\demend
\vspace{-0.5cm}

\subsection{A sufficient condition for \Cref{hyp:DH2OPT}}

\begin{lem}\label{borneINF_hessienne}
For any $\tau \in \overline{\ell_1}(\Lambda)$, 
\begin{equation}\label{borneINF_hess}
D^2 H_\ee(\theta^\ee)[\tau,\tau] \ge \frac{1}{\ee} \left(2 - \int_\YY \int_\XX F^2_{\tht^\ee,y}(x) d\mu(x) d\nu(y)\right)\Vert \tau \Vert_{\ell_2}^2.
\end{equation}
\end{lem}
\begin{proof} We already saw from \eqref{eq:D2H} that for any $\tau \in \overline{\ell_1}(\Lambda)$, 
\begin{eqnarray}
D^2 H_\ee(\theta^\ee)[\tau,\tau]  & = &  \frac{1}{\ee}  \sum\limits_{\lambda' \in \Lambda} \SL \tau_{\lambda'} \overline{\tau_{\lambda}}\int_{\YY } \int_{\XX }\phi_{\lambda'}(x) \overline{\phi_{\lambda}(x)} F_{\tht^\ee,y}(x) d\mu(x)  d\nu(y) \nonumber \\
&- &  \frac{1}{\ee} \int_{\YY } \left\vert \SL \tau_{\lambda}  \int_{\XX } \phi_{\lambda}(x) F_{\tht^\ee,y}(x) d\mu(x)  \right\vert^2 d\nu(y). \label{eq:decompDH}
\end{eqnarray}
Our proof consists in a study of the two terms in the right-hand side of \eqref{eq:decompDH}.  
Since \eqref{GradH} only defines $DH_\ee(\tht^\ee)_\lambda$ for all $\lambda \ne 0$,
we deduce from \eqref{part_deriv_h_ee} and \eqref{GradH} that, for all $\lambda \ne \lambda' $,
\begin{equation*}
\int_{\XX }\phi_{\lambda'} (x)\overline{\phi_{\lambda}(x)} F_{\tht^\ee,y}(x) d\mu(x)  d\nu(y) = 
\int_{\YY } \!\int_{\XX }\phi_{\lambda' - \lambda}(x) F_{\tht^\ee,y}(x) d\mu(x)  d\nu(y)
= DH_\ee(\tht^\ee)_{\lambda'-\lambda}.
\end{equation*}
Moreover, as soon as $\lambda = \lambda'$, 
\begin{equation*}
\int_{\YY } \!\int_{\XX }\phi_{\lambda}(x) \overline{\phi_{\lambda}(x)} F_{\tht^\ee,y}(x) d\mu(x)  d\nu(y) = 
\int_{\YY } \!\int_{\XX } F_{\tht^\ee,y}(x) d\mu(x)  d\nu(y) = 1.
\end{equation*}
Hence, from the optimality condition $DH_\ee(\tht^\ee) = 0$, we obtain that
\begin{equation*}
\int_{\YY } \int_{\XX }\phi_{\lambda'} (x)\overline{\phi_{\lambda}(x)} F_{\tht^\ee,y}(x) d\mu(x)  d\nu(y) = \delta_0(\lambda'-\lambda),
\end{equation*}
where $\delta_0$ stands for the dirac function at $0$. Therefore, it follows that 
\begin{equation}\label{M_tht_ee}
 \frac{1}{\ee}  \sum\limits_{\lambda' \in \Lambda} \SL \tau_{\lambda'} \overline{\tau_{\lambda}}\int_{\YY } \int_{\XX }\phi_{\lambda'}(x) \overline{\phi_{\lambda}(x)} F_{\tht^\ee,y}(x) d\mu(x)  d\nu(y) =  \frac{1}{\ee} \Vert \tau \Vert_{\ell_2}^2.
\end{equation}
From now on, our goal is to find an upper bound for the second term
in the right-hand side of \eqref{eq:decompDH}. By Cauchy-Schwarz's inequality, we have
that
\begin{equation}
\left\vert \SL \tau_{\lambda}  \int_{\XX } \phi_{\lambda}(x) F_{\tht^\ee,y}(x) d\mu(x)  \right\vert^2   \le \Vert \tau \Vert_{\ell_2(\Lambda)}^2 \Vert Dh_\ee(\tht^\ee,y) \Vert_{\ell_2(\Lambda)}^2. \label{eq:CS}
\end{equation}
Moreover, it follows from Parseval's identity, \cite{SteinWeiss}[Theorem 1.7] together with  the fact that $\int_\XX F_{\tht^\ee,y}(x) d\mu(x) = 1$, that
\begin{equation}
 \Vert  Dh_\ee(\tht^\ee,y) \Vert_{\ell_2(\Lambda)}^2 =   \int_\XX F^2_{\tht^\ee,y}(x) d\mu(x) -1. \label{eq:P}
\end{equation}
Hence, combining \eqref{eq:CS} and \eqref{eq:P}, we obtain that 
\begin{equation}\label{N_tht_ee}
 \int_{\YY } \left\vert \SL \tau_{\lambda}  \int_{\XX } \phi_{\lambda}(x) F_{\tht^\ee,y}(x) d\mu(x)  \right\vert^2 d\nu(y) \le \Vert \tau \Vert_{\ell_2}^2 \left( \int_{\YY } \! \int_\XX F^2_{\tht^\ee,y}(x) d\mu(x)d\nu(y) -1\right).
\end{equation}
Finally, we deduce \eqref{borneINF_hess} from \eqref{eq:decompDH}, \eqref{M_tht_ee} and \eqref{N_tht_ee}.
\end{proof}
\vspace{0.1cm}


\subsection{Proof of Theorem \ref{thmAS}} 
We shall proceed to the almost sure convergence of the random  sequence $(\widehat\theta_n)_n$. Let $(V_n)$ be the Lyapunov sequence defined, for all $n\geq 1$, by
\begin{equation*}
V_n = \Vert \widehat\tht_n - \tht^\ee \Vert_{W^{-1}}^2.
\end{equation*}
\Cref{hyp:Winv} ensures that $\| \tht^\ee \Vert_{W^{-1}} < + \infty$. Moreover, we clearly have from
\eqref{defthetan} that
\begin{equation*}
W^{-1/2} \widehat\tht_{n+1} =   W^{-1/2}  \widehat{\theta}_{n} - \gamma_n W^{1/2} D_\theta h_\ee(\widehat{\theta}_{n} , Y_{n+1}),
\end{equation*}
where $W^{\alpha}$ stands for the  linear operator, for $\alpha \in \{-1/2,1/2\}$, that maps $v = (v_\lambda)_{\lambda \in \Lambda} \in \ell_{\infty}(\Lambda)$ to $ (w_\lambda^\alpha v_\lambda)_{\lambda \in \Lambda}$. It follows from \eqref{DEFnormW} that $\|  \widehat\tht_n \Vert_{W^{-1}} = \| W^{-1/2} \widehat\tht_n \|_{\ell_2}$. Consequently,
\begin{equation*}
\|  \widehat\tht_{n+1} \Vert_{W^{-1}} \leq \|  \widehat\tht_{n} \Vert_{W^{-1}} +  \gamma_n \|W^{1/2} D_\theta h_\ee(\widehat{\theta}_{n} , Y_{n+1})\|_{\ell_2}.
\end{equation*}
Furthermore, we obtain from \eqref{eq:opnormgradienth} that
\begin{equation*}
 \|W^{1/2} D_\theta h_\ee(\widehat{\theta}_{n} , Y_{n+1})\|_{\ell_2}^2 \leq  \|w \|_{\ell_1} \sup\limits_{\lambda \in \Lambda} \left\vert \frac{\partial h_\ee(\theta,y) }{\partial \theta_\lambda} \right\vert^2 \leq  \|w \|_{\ell_1} < \infty.
\end{equation*}
Therefore, thanks to the assumption that $\Vert  \widehat\theta_0\Vert_{W^{-1}}  < +\infty$, we deduce by induction that $\|  \widehat\tht_n \Vert_{W^{-1}} < + \infty$, which means that the  Lyapunov sequence $(V_n)$ is well defined. From now on, it follows from \eqref{defthetan}  and \eqref{DEFnormW} that for all $n \geq 0$,
\begin{eqnarray*}
V_{n+1} &=& \Vert  \widehat\tht_n - \tht^\ee - \gamma_n W D_\theta h_\ee(\widehat{\theta}_{n} , Y_{n+1}) \Vert_{W^{-1}}^2,\\
&=& V_n - 2\gamma_n \langle \widehat\tht_n - \tht^\ee, D_\theta h_\ee(\widehat\tht_n, Y_{n+1})  \rangle + \gamma_n^2 \Vert D_\theta h_\ee(\widehat\tht_n, Y_{n+1}) \Vert_{W}^2.
\end{eqnarray*}
Moreover, \eqref{eq:opnormgradienth} implies that $\Vert D_\theta h_\ee(\widehat\tht_n, Y^{n+1}) \Vert_{W}^2 \le \Vert w \Vert_{\ell_1}$ which ensures that
for all $n \geq 0$,
\begin{equation}
\label{RS1}
V_{n+1} \leq V_n - 2\gamma_n \langle \widehat\tht_n - \tht^\ee, D_\theta h_\ee(\widehat\tht_n, Y_{n+1})  \rangle + \gamma_n^2 \Vert w \Vert_{\ell_1}.
\end{equation}
Denote by $\mathcal F_n= \sigma(Y_1,\ldots,Y_n)$ the $\sigma$-algebra generated by $Y_1,\ldots,Y_n$ drawn from $\nu$. 
From \Cref{prop:grad}, $\mathbb E[D_\theta h_\ee(\widehat\tht_n, Y_{n+1}) | \mathcal F_n]= D H_\ee(\widehat\tht_n)$,
which implies via \eqref{RS1} that for all $n \geq 0$,
\begin{equation}
\mathbb E[V_{n+1} \vert \mathcal F_n ] \le V_n +A_n - B_n
\hspace{1cm}\text{a.s.}
\end{equation}
where $(A_n)$ and $(B_n)$ are the two positive sequences given, for all $n\geq 0$, by
\begin{equation*}
A_n= \gamma_n^2\Vert w \Vert_{\ell_1} \hspace{1cm}\text{and}\hspace{1cm}
B_n= 2\gamma_n DH_\ee(\widehat\tht_n)[\widehat\tht_n - \tht^\ee].
\end{equation*}
Therefore, as
$\sum_{n=0}^\infty A_n < \infty,
$
we deduce from the Robbins-Siegmund theorem \cite{robbins_monro}
that the sequence $ (V_n)$ converges almost surely to a finite random variable 
$V$ and that the series
\begin{equation}\label{cvgBn}
\sum_{n=0}^\infty B_n = 2 \sum_{n=0}^\infty  \gamma_n DH_\ee(\widehat\tht_n)[\widehat\tht_n - \tht^\ee]  < \infty 
\hspace{1cm}\text{a.s.}
\end{equation}
Hence, by combining \eqref{cvgBn} with  the first condition in \eqref{condgamma}, it necessarily follows that
\begin{equation}
\lim_{n \to \infty} DH_\ee(\widehat\tht_n)[\widehat\tht_n - \tht^\ee] = 0  \hspace{1cm} \text{a.s.} \label{eq:convDH}
\end{equation}
Hereafter, our goal is to prove that $\Vert \widehat\tht_n - \tht^\ee \Vert_{\ell^2}$ goes to zero almost surely as $n$ tends to infinity. From now on, let $g$ be the function defined in \eqref{defg}.
One can easily see that $g$ is a continuous and strictly decreasing function. Moreover, using the Cauchy-Schwarz inequality, one has that $\Vert \widehat\tht_n - \tht^\ee \Vert_{\ell_1}^2\leq \Vert w\Vert_{\ell_1} V_n$. Hence, it follows from inequality 
\eqref{eq:Grad_vs_hess} that for all $n \geq 0$,
\begin{equation}
\label{RS2}
DH_\ee(\widehat\tht_n)[\widehat\tht_n - \tht^\ee] \geq g \Big(  \frac{2}{\varepsilon} \Vert w\Vert_{\ell_1} V_n \Big) D^2H_\ee(\tht^\ee)[\widehat\tht_n-\tht^\ee,\widehat\tht_n-\tht^\ee].
\end{equation}
Therefore, we obtain from \Cref{hyp:DH2OPT}  and inequality \eqref{RS2} that 
for all $n \geq  0$,
\begin{equation}
\label{RS3}
DH_\ee(\widehat\tht_{n})[\widehat\tht_{n} - \tht^\ee] \geq c_\ee g \Big(  \frac{2}{\varepsilon} \Vert w\Vert_{\ell_1} V_{n} \Big)  \Vert \widehat\tht_{n}-\tht^\ee \Vert_{\ell^2}^2.
\end{equation}
Since $(V_n)$ converges a.s.\ to a finite random variable $V$, it follows by continuity of $g$ that  
\begin{equation}
\label{RS4}
\lim_{n \to \infty}  g \Big(  \frac{2}{\varepsilon} \Vert w\Vert_{\ell_1} V_n \Big)  =  g \Big(  \frac{2}{\varepsilon} \Vert w\Vert_{\ell_1} V  \Big) 
\hspace{1cm}\text{a.s.}
\end{equation}
and the limit in the right-hand side of \eqref{RS4} is positive almost surely. 
Therefore, we conclude from \eqref{eq:convDH}, \eqref{RS3} and \eqref{RS4} that
\begin{equation*}
\lim_{n \to \infty}\Vert \widehat\theta_n - \theta^{\ee}   \Vert_{\ell_2} = 0 \hspace{1cm} \text{a.s.}
\end{equation*}
Finally, we deduce from Parseval's identity, \cite{SteinWeiss}[Theorem 1.7] that 
\begin{equation*}
\int_{\XX}|\widehat{u}_{\varepsilon}^{\, n}(x)  - u_\ee(x)|^2 d\mu(x) = \Vert \widehat\theta_n - \theta^{\ee}   \Vert_{\ell_2}^2
\end{equation*}
which achieves the proof of \Cref{thmAS}.

\demend
\vspace{-0.5cm}

\newpage
$$
\textbf{\Large{SUPPLEMENTARY MATERIALS}}
$$

\appendix
\renewcommand{\thesection}{SM.\Alph{section}}
\pagenumbering{roman}
\section{Proof of Proposition \ref{prop:grad}}\label{sup:prop:grad}

We first state a result about Fr\'echet differentiation under Lebesgue integrals, that follows from \cite[Lemma A.2]{LeibnizRule}, and which extends well-known results on the differentiation of integral functionals. For the proof of a similar result, we also refer to the unpublished note \cite{Kammar}. 

\begin{lem} [Leibniz's rules of Fr\'echet differentiation]\label{lem:Leibniz}
\hspace{-0.4cm} Let $(\Theta,\| \cdot \|)$ be an infinite dimensional Banach space and $\sigma$ a finite measure on a measurable space $\TT$. Let $\theta_{0} \in \Theta$ and denote by $B(\theta_{0},R) \subset \Theta$ the ball  of center $\theta_{0}$ and radius $R$.  Consider a function $f : \Theta \times \TT \to \R$ that is Fr\'echet differentiable at $\theta_{0}$ (for every $t \in \TT$), and suppose that there exists $K \in L^{1}(\sigma)$
such that, for all $\theta_1, \theta_2 \in B(\theta_{0},R)$ and all $ t \in \TT$,
\begin{equation*}
|f(\theta_1,t) - f(\theta_2,t) | \leq K(t) \| \theta_1 - \theta_2 \|.
\end{equation*}
Then, the integral functional $F : \Theta \to \R$ defined by
$
F(\theta) = \int_{\TT} f(\theta,t) d\sigma(t)
$
is  Fr\'echet differentiable at $\theta_{0}$ and
\begin{equation*}
D F(\theta_{0}) = \int_{\TT}  D_{\theta} f(\theta_{0},t) d\sigma(t),
\end{equation*}
where  $D_{\theta} f(\theta_{0},t)$ denotes the Fr\'echet derivative of $\theta \mapsto f(\theta,t)$ at $\theta_{0}$.
\end{lem}
In what follows, we will apply \cref{lem:Leibniz} with $\Theta = \bar{\ell}_{1}(\Lambda)$, $\TT = \XX$ and $\sigma = \mu$ to obtain the expression of the Fr\'echet differential of $H_\ee$.  Let us  first prove that, for every $y \in \YY$, the function $h_\ee (\cdot,y)  : \bar{\ell}_{1}(\Lambda) \to \R$ defined in \eqref{Se}
is Fr\'echet differentiable.  To this end, we introduce the function $g_{y}(\cdot,x) :  \bar{\ell}_{1}(\Lambda) \to \R$ defined as
\begin{equation}
\label{def_gy_tht}
g_{y}(\theta,x) = \frac{1}{\ee} \left(\SL \theta_\lambda \phi_\lambda(x) - c(x,y) \right)
\end{equation}
and $G_{y}(\tht) = \int_\XX \exp (g_{y}(\theta,x)) d\mu(x)$. In this  way,  one has that $h_\ee (\theta,y)  = \ee \log G_{y}(\tht)+\ee$. For every $x \in \XX$, the function $\theta \mapsto \exp(g_{y}(\theta,x)) $ is clearly Fr\'echet differentiable and, for $ \tau \in \bar{\ell}_{1}(\Lambda)$,
\begin{equation}
\label{deriv_gy_tht}
D_{\theta} \exp(g_{y}(\theta,x))[\tau] =  \frac{1}{\varepsilon} \SL  \exp(g_{y}(\theta,x))  \phi_\lambda(x)  \tau_{\lambda}.
\end{equation}
Moreover, it is a bounded linear operator from $\bar{\ell}_{1}(\Lambda)$ to $\R$. In what follows, we identify this operator to the infinite-dimensional vector
\begin{equation*}
D_{\theta} \exp(g_{y}(\theta,x)) =  \frac{1}{\varepsilon}  \exp(g_{y}(\theta,x)) \left( \overline{\phi_\lambda(x)}  \right)_{\lambda \in \Lambda}.
\end{equation*}
From now on, let $\theta_0 \in  \bar{\ell}_{1}(\Lambda)$ and $R > 0$. Then, for any $\theta_1,\theta_2 \in B(\theta_{0},R)$, the mean value theorem for functions defined on a Banach space implies that
\begin{equation}
| \exp (g_y(\theta_1,x)) - \exp (g_y(\theta_2,x)) | \leq \sup_{\theta \in B(\theta_{0},R)} \|D_{\theta} \exp (g_{y}(\theta,x)) \|_{op} \|\theta_1  - \theta_2 \|, \label{eq:meanvalue}
\end{equation}
where the operator norm of $D_{\theta} \exp (g_{y}(\theta,x))$ is defined as
\begin{equation*}
 \|D_{\theta} \exp (g_{y}(\theta,x)) \|_{op} = \sup_{\| \tau \|_{\ell_1} \leq 1} \vert D_{\theta} \exp (g_{y}(\theta,x)) [\tau]\vert .
\end{equation*}
Since
\begin{equation*}
 | D_{\theta} \exp (g_{y}(\theta,x)) [\tau]|  = \left| \frac{1}{\varepsilon} \SL   \exp (g_{y}(\tht,x))   \phi_\lambda(x) \tau_{\lambda} \right| \leq \frac{1}{\varepsilon} \exp ( g_{y}(\tht,x))  \SL | \tau_{\lambda} |,
\end{equation*}
one has that, for any $\theta \in \bar{\ell}_{1}(\Lambda)$,
\begin{equation}
 \|D_{\theta} \exp (g_{y}(\theta,x)) \|_{op} \leq  \frac{1}{\varepsilon} \exp ( g_{y}(\tht,x)) \leq   \frac{1}{\varepsilon} \exp\left(\frac{\SL \vert \theta_\lambda \vert + c(x,y)}{\ee}\right) \label{eq:op}
\end{equation}
Consequently, let
\begin{equation*}
K_y(x) =  \frac{1}{\varepsilon} \exp\left( \frac{c(x,y)}{\ee} \right) \sup_{\theta \in B(\theta_{0},R)} \exp\left( \frac{\| \theta \|_{\ell_1}}{\ee} \right).
\end{equation*}
It follows from \eqref{eq:meanvalue} and \eqref{eq:op} that, for all $ \theta_1,\theta_2 \in B(\theta_{0},R)$  and $ x \in \XX$,
\begin{equation*}
| \exp (g_{y}(\theta_1,x)) - \exp (g_{y}(\theta_2,x)) | \leq  K_y(x)  \|\theta_1  - \theta_2 \|.
\end{equation*}
Obviously, for all $y\in\YY$, the function $K_y$ belongs to $L^{1}(\mu)$, and therefore, by  \cref{lem:Leibniz}, we conclude that $G_{y}$  and  $ h_\ee (\cdot,y)$ are Fr\'echet differentiable  and that the linear operator $D_{\theta} h_\ee (\theta,y)$ is identified as an element of $\bar{\ell}_{\infty}(\Lambda)$ given by
\begin{eqnarray}
\label{Dgrad_h_ee}
D_{\theta} h_\ee (\theta,y) & =  & \left( \frac{\int_\XX  \overline{\phi_\lambda(x)} \exp(g_{y}(\tht,x)) d\mu(x) }{ \int_\XX \exp(g_{y}(\tht,x))d\mu(x)}\right)_{\lambda \in \Lambda}.
\end{eqnarray}
Similarly, to prove that the function $H_\ee(\theta) = \int_{\YY}  h_\ee (\theta,y) d\nu(y)$ is Fr\'echet differentiable, it is sufficient to bound the operator norm of $D_{\theta} h_\ee (\theta,y)$ for $\theta \in B(\theta_{0},R)$. Recalling that
\begin{equation*}
F_{\theta,y}(x) = \frac{\exp \left(  g_{y}(\tht,x) \right) }{ \int_\XX \exp \left(g_{y}(\tht,x) \right) d\mu(x)},
\end{equation*}
we remark that, for any  $\tau \in  \bar{\ell}_{1}(\Lambda)$,
\begin{equation*}
|D_{\theta} h_\ee (\theta,y) [\tau]| = \left| \SL \int_\XX F_{\theta,y}(x)  \phi_\lambda(x)d\mu(x) \tau_{\lambda} \right| \leq  \int_\XX F_{\theta,y}(x) d\mu(x)  \SL  |\tau_{\lambda}| = \| \tau \|_{\ell_1}.
\end{equation*}
Therefore, $\| D_{\theta} h_\ee (\theta,y) \|_{op} \leq 1$ which proves inequality \eqref{eq:opnormgradienth}. It also means that $D_{\theta} h_\ee(\theta, y)$ can be identified as an element of $\bar{\ell}_{\infty}(\Lambda)$. Thus, arguing as previously, that is by combining the mean value theorem with \cref{lem:Leibniz}, we obtain that $H_\ee(\theta)$ is Fr\'echet differentiable with
\begin{equation*}
D H_\ee(\theta) =  \int_{\YY}  D_{\theta} h_\ee (\theta,y) d\nu(y) = \left( \int_{\YY} \frac{\partial h_\ee(\theta,y) }{\partial \theta_\lambda}   d \nu(y) \right)_{\lambda \in \Lambda}
\end{equation*}
which can also be identified as an element of  $\bar{\ell}_{\infty}(\Lambda)$ such that $\| D H_\ee(\theta) \|_{op} = \| D H_\ee(\theta) \|_{\ell_{\infty}}$ satisfies inequality \eqref{eq:opnormgradientH} by combining inequality \eqref{eq:opnormgradienth} together with the fact that $\nu$ is a probability measure. This achieves the proof of \cref{prop:grad}.

\hfill
\demend

\section{Proof of Proposition \ref{prop:Hess}}\label{sup:prop:Hess}
First, let us recall that, for $(\Theta,\| \cdot \|)$ a given  Banach space, a function $f : \Theta \to \R$  is twice Fr\'echet differentiable if $D f $ is  Fr\'echet differentiable. In this case, the second order  Fr\'echet derivative of $f$ at $\theta_0$ is denoted by $ D^2 f(\theta_0)$ and it is identified as an element of  $L(\Theta \times \Theta, \R)$ the set of continuous bilinear mapping  from $\Theta \times \Theta$ to $\R$. Moreover, the operator norm of $ D^2 f(\theta_0)$ is defined as
\begin{equation*}
\| D^2 f(\theta_0) \|_{op} = \sup_{\|\theta\| \leq 1, \| \theta'\| \leq 1} |D^2 f(\theta_0)[\theta,\theta']|.
\end{equation*}
To derive the expression of the  second order  Fr\'echet derivative of the functions $h_\ee(\cdot,y)$ and $H_{\varepsilon}$, we use similar arguments to those in the proof of \cref{prop:grad}. First, recall from \eqref{Dgrad_h_ee} that the Fr\'echet derivative $D_\tht h_\ee(\theta,y)$ is the linear operator defined as
\begin{equation*}
D_\tht h_\ee(\theta,y) = \left( \int F_{\theta,y}(x) \overline{\phi_\lambda(x)} d\mu(x) \right)_{\lambda \in \Lambda} =  \int  \psi_{y}(x,\theta) d\mu(x) 
\end{equation*}
where $\psi_{y}(x,\theta) :  \bar{\ell}_{1}(\Lambda) \to \R$ is  the linear operator
\begin{align*}
\psi_{y}(x,\theta)[\tau] = \SL  F_{\theta,y}(x) \phi_\lambda(x) \tau_{\lambda} = \SL  F_{\theta,y}(x) \overline{\phi_\lambda(x)} \overline{\tau_{\lambda}}.
\end{align*}
As a standard strategy, we aim to derive this with respect to $\tht$. 
From \eqref{def_gy_tht}, one has that 
\begin{equation*}
F_{\theta,y}(x) = \frac{\exp(g_y(\tht,x))}{G_y(\tht)}.
\end{equation*}
Therefore, using \eqref{deriv_gy_tht} combined with the differentiability of $G_y(\tht)$,
\begin{equation}\label{dtht_Ftheta}
D_\tht F_{\tht,y}(x)[\tau] = \frac{1}{\ee} \left( \SL \tau_\lambda \phi_\lambda(x) F_{\tht,y}(x) - F_{\tht,y}(x)\int_\XX \SL \tau_\lambda \phi_\lambda(z)F_{\tht,y}(z)d\mu(z) \right).
\end{equation}
Thus, the mapping $\theta \mapsto \psi_{y}(x,\theta)[\tau]$ is clearly Fr\'echet differentiable, and its Fr\'echet derivative can be identified as the following symmetric  bilinear mapping from $\bar{\ell}_{1}(\Lambda)  \times \bar{\ell}_{1}(\Lambda) $ to $\R$
\begin{eqnarray*}
D_{\theta} \psi_{y}(x,\theta)[\tau,\tau']  & =  & \frac{1}{\ee} \left( \sum\limits_{\lambda' \in \Lambda} \SL \tau'_{\lambda'} \overline{\tau_{\lambda}} \left( \phi_{\lambda'}(x)\overline{\phi_{\lambda}(x)} F_{\theta,y}(x) -   \phi_{\lambda'}(x)F_{\theta,y}(x)   \overline{\phi_{\lambda}(x)} F_{\theta,y}(x)  \right) \right) \\
 & = &  \frac{1}{\ee}  \left( \sum\limits_{\lambda' \in \Lambda} \SL \tau'_{\lambda'} \overline{\tau_{\lambda}} \phi_{\lambda'}(x)\overline{\phi_{\lambda}(x)} F_{\theta,y}(x) - \SL \tau'_{\lambda} \phi_{\lambda}(x)F_{\theta,y}(x)   \overline{ \SL \tau_{\lambda} \phi_{\lambda}(x)F_{\theta,y}(x) } \right).
\end{eqnarray*}
We now compute an upper bound for the norm of this linear operator. One can observe that, for $\tau = \tau'$,
\begin{equation}
|D_{\theta} \psi_{y}(x,\theta)[\tau,\tau]| \leq \frac{1}{\ee} F_{\theta,y}(x)  \|\tau\|_{\ell_1}^2,   \label{eq:boundtau}
\end{equation}
thanks to the elementary fact that
\begin{equation}
\label{Hess_vert2_pos}
\SL \tau_{\lambda} \phi_{\lambda}(x)F_{\theta,y}(x)   \overline{ \SL \tau_{\lambda} \phi_{\lambda}(x)F_{\theta,y}(x) } =  \left|  \SL \tau_{\lambda} \phi_{\lambda}(x)F_{\theta,y}(x) \right|^2 \geq 0.
\end{equation}
Then, using the equality,
\begin{equation*}
4 D_{\theta} \psi_{y}(x,\theta)[\tau,\tau'] = D_{\theta} \psi_{y}(x,\theta)[\tau+\tau',\tau+\tau'] -D_{\theta} \psi_{y}(x,\theta)[\tau-\tau',\tau-\tau'] 
\end{equation*}
combined with the upper bound \eqref{eq:boundtau}, we obtain that
\begin{equation*}
4 |D_{\theta} \psi_{y}(x,\theta)[\tau,\tau']|  \leq  \frac{1}{\ee} F_{\theta,y}(x) \left( \|\tau +\tau'\|_{\ell_1}^2 +  \|\tau -\tau'\|_{\ell_1}^2  \right).
\end{equation*}
Therefore, we immediately obtain that
\begin{equation*}
\sup_{\| \tau \|_{\ell_1} \leq 1, \| \tau' \|_{\ell_1} \leq 1} |D_{\theta} \psi_{y}(x,\theta)[\tau,\tau']| \leq  \frac{2}{\ee} F_{\theta,y}(x).
\end{equation*}
Consequently, 
we may proceed as in the proof of \cref{prop:grad} to obtain that $h_\ee(\cdot,y)$ is twice Fr\'echet differentiable and that its second Fr\'echet derivative is  the following symmetric  bilinear mapping  from $\bar{\ell}_{1}(\Lambda)  \times \bar{\ell}_{1}(\Lambda) $ to $\R$
\begin{eqnarray*}
D^2_\tht h_\ee(\theta,y)[\tau,\tau']  & = &  \frac{1}{\ee}  \sum\limits_{\lambda' \in \Lambda} \SL \tau'_{\lambda'} \overline{\tau_{\lambda}} \int_{\XX }\phi_{\lambda'}(x) \overline{\phi_{\lambda}}(x) F_{\theta,y}(x) d\mu(x) \\
& & -  \frac{1}{\ee} \left( \SL \tau'_{\lambda} \int_{\XX } \phi_{\lambda}(x) F_{\theta,y}(x) d\mu(x)  \right) \overline{\left( \SL \tau_{\lambda} \int_{\XX } \phi_{\lambda}(x)F_{\theta,y}(x) d\mu(x) \right)}.
\end{eqnarray*}
Note that, for $\tau = \tau'$, an application of Jensen's inequality with respect to the probability measure $F_{\theta,y}(x) d\mu(x) $ implies that $D^2_\tht h_\ee(\theta,y)[\tau,\tau] \geq 0$. Moreover, 
it follows once again from \eqref{Hess_vert2_pos} together with
the elementary fact that $ \int_{\XX} F_{\theta,y} d\mu = 1$, that
\begin{equation}
D^2_\tht h_\ee(\theta,y)[\tau,\tau] \leq \frac{1}{\ee} \|\tau\|_{\ell_1}^2.   \label{eq:boundtau2}
\end{equation}
Hereafter, we deduce from the equality
\begin{equation*}
4 D^2_\tht h_\ee(\theta,y) [\tau,\tau'] = D^2_\tht h_\ee(\theta,y)[\tau+\tau',\tau+\tau'] -D^2_\tht h_\ee(\theta,y)[\tau-\tau',\tau-\tau'], 
\end{equation*}
the positivity of $D^2_\tht h_\ee(\theta,y)[\tau-\tau',\tau-\tau']$ and inequality \eqref{eq:boundtau2}, that 
 \begin{equation*}
 4 |D^2_\tht h_\ee(\theta,y) [\tau,\tau']| \leq D^2_\tht h_\ee(\theta,y)[\tau+\tau',\tau+\tau'] \leq  \frac{1}{\ee} \|\tau +\tau'\|_{\ell_1}^2.
 \end{equation*}
It ensures that
\begin{equation*}
\| D^2_\tht h_\ee(\theta,y) \|_{op} = \sup_{\| \tau \|_{\ell_1} \leq 1, \| \tau' \|_{\ell_1} \leq 1} |D^2_\tht h_\ee(\theta,y) [\tau,\tau'] | \leq  \frac{1}{\ee},
\end{equation*}
which proves inequality \eqref{eq:opnormHessh}. 

Finally, combining the above upper bound on $\| D^2_\tht h_\ee(\theta,y) \|_{op} $ and using again  an adaptation of \cref{lem:Leibniz} to obtain a Leibniz's formula for the second order Fr\'echet differentiation under the integral sign, one can prove that $H_\ee(\theta)$ is twice Fr\'echet differentiable by integrating $D^2_\tht h_\ee(\theta,y) $ with respect to  $d\nu(y)$, which implies that $D^2  H_\ee(\theta)$ is the linear operator defined by \eqref{eq:D2H}. Moreover, the upper bound \eqref{eq:opnormHessH} follows from inequality \eqref{eq:opnormHessh} and the fact that $\nu$ is a probability measure, which completes the proof of \cref{prop:Hess}.

\hfill
\demend

\section{Proof of Proposition \ref{prop:convexity}} \label{sup:prop:convexity}

For $x \in \XX$, $y\in\YY$ and for $ (\theta^{(1)},\theta^{(2)}) \in \bar{\ell}_{1}(\Lambda)\times \bar{\ell}_{1}(\Lambda)$, denote
\begin{equation*}
e_1(x,y) = \exp \left( \frac{\SL \theta_\lambda^{(1)} \phi_\lambda(x) -c(x,y)}{\epsilon} \right)
\end{equation*}
and
\begin{equation*}
e_2(x,y) =  \exp \left( \frac{\SL \theta_\lambda^{(2)} \phi_\lambda(x) -c(x,y)}{\epsilon} \right).
\end{equation*}
We have, for all $0 < t < 1$, and for a fixed $y\in\YY$, that 
\begin{align}
t h_\ee ( \theta^{(1)}, y ) + (1-t)h_\ee ( \theta^{(2)}, y)  
& = \ee t \log \int_{\XX} e_1(x,y) d\mu(x) + (1-t) \log \int_{\XX} e_2(x,y) d\mu(x)   +  \ee, \nonumber \\
& = \ee \mathbb  \log \left( \left( \int_{\XX} e_1(x,y) d\mu(x) \right)^t \left( \int_{\XX} e_2(x,y) d\mu(x) \right)^{1-t}\right) + \ee. \label{eq:log}
\end{align}
Hereafter, applying H{\"o}lder's inequality to the functions $f(x)=e_1^t(x,y)$ and $g(x)=e_2^{1-t}(x,y)$ with H{\"o}lder conjugates $p=1/t$ and $q=1/(1-t)$, we obtain that
\begin{equation}
\int_{\XX} f(x)g(x) \; d\mu(x) \le \left( \int_{\XX} e_1(x,y) d\mu(x) \right)^t \left( \int_{\XX} e_2(x,y) d\mu(x) \right)^{1-t}.
 \label{eq:Holder}
\end{equation}
However, one can observe that
\begin{align*}
\int_{\XX} f(x)g(x) \; d\mu(x) 
&= \int_{\XX} \exp \left( \frac{ t\SL \theta_\lambda^{(1)} \phi_\lambda(x) + (1-t) \SL \theta_\lambda^{(2)} \phi_\lambda(x)  + c(x,y)}{\ee} \right) d\mu(x),
\end{align*}
which ensures that
\begin{equation*}
\ee   \log\int_{\XX} f(x)g(x)\; d\mu(x) + \ee = h_\ee ( t\theta^{(1)} + (1-t)\theta^{(2)} , y).
\end{equation*}
Hence,  combining the above equality with  \eqref{eq:log} and \eqref{eq:Holder}, we obtain that
\begin{equation*}
h_\ee ( t\theta^{(1)} + (1-t)\theta^{(2)}, y ) \le t h_\ee ( \theta^{(1)}, y ) + (1-t)h_\ee ( \theta^{(2)}, y ),
\end{equation*}
which proves the convexity of $\theta \mapsto h_\ee(\theta,y)$.  Since  $H_\ee(\theta) = \EE \left[  h_{\ee}(\theta,Y) \right]$, we also obtain the convexity of the function $H_\ee$.
Furthermore, assume that H{\"o}lder's inequality \eqref{eq:Holder} becomes an equality, which means that the functions $f^p$ and $g^q$ are linearly dependent in $L^1(\mu)$. This would mean that 
it exists $\beta_y>0$ such that $e_1(x,y) = \beta_y e_2(x,y)$ for all $x\in\XX$. Applying the logarithm, this equality is equivalent to
\begin{equation*}
\frac{1}{\ee}\SL \theta_\lambda^{(1)} \phi_\lambda(x)  = \log \beta_y + \frac{1}{\ee}\SL \theta_\lambda^{(2)} \phi_\lambda(x).
\end{equation*}
By integrating the above equality with respect to $\mu$, and from our normalization condition \eqref{eq:normcond},
the equality case in the H{\"o}lder inequality  \eqref{eq:Holder}  implies that $\log \beta_y = 0$ and thus $\beta_y = 1$. But then one has that $e_1(x,y)=e_2(x,y)$, implying that $\SL (\theta_\lambda^{(1)} - \theta_\lambda^{(2)}) \phi_\lambda(x) = 0$   for all $x \in \XX$. Hence, we necessarily have that $\theta^{(1)} = \theta^{(2)}$ which yields a contradiction. Therefore, the function $\theta \mapsto h_{\ee}(\theta,y)$ is strictly convex. Since  $H_\ee(\theta) = \EE \left[  h_{\ee}(\theta,Y) \right] $ this also implies the strict convexity of  $H_\ee $, which achieves the proof of \cref{prop:convexity}.

\hfill
\demend

\section{Optimal transport with periodicity constraints} \label{sec:periodic}

In this section, we focus our attention on conditions such that the dual potential \eqref{eq:dualOT} is periodic at the boundary of $\XX$.

\subsection{The case of the standard quadratic cost}

Using classical results in  the analysis of multiple Fourier series (see e.g.\ \cite[Corollary 1.8]{SteinWeiss}), assuming that the Fourier coefficients $\theta^0 =  (\theta^0_\lambda)_{\lambda \in \Lambda}$ of $u_0$ form an absolutely convergent series implies that $u_0$ can be extended as a continuous and $\ZZ^d$-periodic function on $\R^d$.  Hence, under this assumption, $u_0$ has to be a continuous function that is  constant at the boundary of $\XX = [0,1]^d$. However, for the quadratic cost  $c(x,y) = \frac{1}{2} \|x-y\|^2$, we are not aware of standard results on the regularity of optimal transport (through smoothness assumptions on $\nu$) that would imply periodic properties of $u_0$ and its derivatives at the boundary of $\XX$.

\subsection{The  quadratic cost on the torus}

Nevertheless, guaranteeing the periodicity of $u_0$ and the summability of its Fourier coefficients is  feasible by considering the setting $\XX = \YY = \TT^d$, where $ \TT^d = \R^d/\ZZ^d$ is the $d$-dimensional torus, that is endowed with the usual distance
$$
d_{\TT^d}(x,y) = \min_{\lambda \in \ZZ^d} \|x-y + \lambda \|.
$$
Hereafter, we identify the torus  as the set of equivalence classes $\{x + \lambda \; : \; \lambda \in \ZZ^d \}$ for $x \in [0,1)^d$, and we use the notation $[x]=x+\lambda_{0}$ where $\lambda_{0} \in \ZZ^d$ is such that $\|x + \lambda\|$ is minimal for $\lambda \in  \ZZ^d$. We also recall that a function $u : \TT^d \to \R$ can be identified as a $\ZZ^d$-periodic function on $\R^d$.  Finally, one can observe that for a given $y \in \TT^{d}$, the cost function $c(x,y) = \frac{1}{2} d_{\TT^d}^{2}(x,y)$ is almost everywhere differentiable, and its gradient is (see e.g.\  \cite[Section 1.3.2]{santambrogio2015optimal})
$$
\nabla_x c(x,y) = [x-y],
$$
at every $x \notin  y +  \{\partial \Omega + \ZZ^d \}$ where $ \partial \Omega$ denotes the boundary of $\Omega =[-\frac{1}{2},\frac{1}{2}]^d$. 

Assuming that the probability measure $\nu$ is also supported on  the $d$-dimensional torus $\TT^d$ allows to use existing results for optimal transport on the torus (see e.g.\  \cite{CORDEROERAUSQUIN1999199},  \cite[Section 2.2]{Manole2021} and \cite[Section 1.3.2]{santambrogio2015optimal}).  Formally, taking $\XX = \YY = \TT^d$ implies that $\nu$ is considered as a periodic positive Radon measure on $\R^d$ with $\nu(\TT^d) = 1$, and that $\mu$ is understood as the Lebesgue measure on $\R^d$. Note that this setting is not restrictive, as it allows to treat the example of an absolutely continuous measure $\nu$ with support on $[0,1]^d$ whose density $f_{\nu}$ takes a constant value on the boundary of  $[0,1]^d$, implying that $f_{\nu}$ can be extended over $\R^d$ as a $\ZZ^d$-periodic function.

 Then, thanks to the identification of $u : \TT^d \to \R$ as a $\ZZ^d$-periodic function on $\R^d$, it follows that
$$
\inf_{x \in \TT^d} \left\{\frac{1}{2} d_{\TT^d}^{2}(x,y)  - u(x) \right\}  =  \inf_{x \in \R^d} \left\{\frac{1}{2} \|x-y\|^2  - u(x) \right\}.
$$
Therefore, it is equivalent to define the conjugate of a function $u : \TT^d \to \R$ with respect  to the cost $c(x,y) = \frac{1}{2} d_{\TT^d}^{2}(x,y)$ or to the quadratic cost  $c(x,y) = \frac{1}{2} \|x-y \|^2$ using the periodization of $u$ over $\R^d$. Now, using results on optimal transport on $\TT^d$, previously established in \cite{CORDEROERAUSQUIN1999199} or \cite[Proposition 4]{Manole2021}, it follows that 
\begin{enumerate}
\item[(i)]  there exists a unique optimal transport map $Q : \TT^d \to \TT^d$ from $\mu$ to $\nu$ such that
$$
Q =\argminE_{T \; : \; T \# \mu = \nu} \EE \left(  d_{\TT^d}^{2}(X,T(X)) \right),
$$
\item[(ii)] $Q(x) = x - \nabla u_{0}(x)$ where  $u_{0}$ is a $\ZZ^d$-periodic function on $\R^d$ that is a solution of the  dual  problem \eqref{eq:dualOT}  with $\XX = \YY = \TT^d$ and $c(x,y) = \frac{1}{2} d_{\TT^d}^{2}(x,y) $,
\item[(iii)]  $\| Q(x) - x\|^2 = d_{\TT^d}^{2}(x,Q(x))$  for almost every $x \in \R^d$. \\
\end{enumerate} 

Entropically regularized optimal transport on the torus has also been recently considered in  \cite{Berman2020} and \cite[Section E]{chizat2020sinkhorn}. One can thus also consider the dual formulation of entropic OT as in \eqref{eq:dualOTreg} with the cost $c(x,y) = \frac{1}{2} d_{\TT^d}^{2}(x,y) $. 


 We conclude this section on optimal transport on the torus by a discussion on the regularity of the optimal dual functions in the un-regularized case $\ee = 0$. For $s \in \NN$, we denote by $\mathcal{C}^s(\TT^d)$, the set of $\ZZ^d$-periodic functions $f$ on $\R^d$ having everywhere defined continuous  partial derivatives.
 Then, the following regularity result holds as an immediate application of results from  \cite{CORDEROERAUSQUIN1999199} and \cite[Theorem 5]{Manole2021}.
 
%

\begin{lem} \label{lem:ctransform}
Let $u_{0}$ be a a solution of the  dual  problem \eqref{eq:dualOT}  with $\XX = \YY = \TT^d$ and  $c(x,y) = \frac{1}{2} d_{\TT^d}^{2}(x,y) $. Suppose that the probability distribution $\nu$ is absolutely continuous with a density $f_{\nu}$ that is lower and upper bounded by positive constants. Assume further that   $f_{\nu} \in \mathcal{C}^{s-1}(\TT^d)$ for some $s > 1$. Then,  $u_{0}$ belongs to $\mathcal{C}^{s+1}(\TT^d)$.
\end{lem}

Consequently, under the assumptions of \cref{lem:ctransform}, one has that if $f_{\nu} \in \mathcal{C}^{s-1}(\TT^d)$ for some $s > d/2 -1$, then $u_{0}$ belongs to $\mathcal{C}^{k}(\TT^d)$ with $k > d/2$. Therefore, by standard results for multiple Fourier series  (see e.g.\ \cite[Corollary 1.9]{SteinWeiss}), one has that $\sum_{\lambda \in \Lambda} |\theta_\lambda^{0} | < + \infty$. Hence we can conclude that if the density of $\nu$ is sufficiently smooth (and is upper and lower bounded by positive constants), then the Fourier serie of $u_0$ actually belongs to $\ell_1(\Lambda)$.


\bibliographystyle{siamplain}
\bibliography{RegularizedOT_Quantiles.bib}

\end{document}